\documentclass[a4paper, reqno, 11pt]{amsart}

\usepackage[english]{babel}
\usepackage{amsmath}
\usepackage{mathrsfs}
\usepackage{amssymb}
\usepackage{amsthm}
\usepackage{enumerate}
\usepackage{ifthen}
\usepackage{color}
\usepackage{xcolor}
\usepackage{graphicx}
\provideboolean{shownotes} 
\setboolean{shownotes}{true}
\usepackage{hyperref}
\usepackage{enumitem}
\usepackage{geometry}
\newcommand{\margnote}[1]{
\ifthenelse{\boolean{shownotes}}%
{\marginpar{\raggedright\tiny\texttt{#1}}}%
{}%
}

\newcommand{\hole}[1]{
\ifthenelse{\boolean{shownotes}}%
{\begin{center} \fbox{ \rule {.25cm}{0cm}
\rule[-.1cm]{0cm}{.4cm} \parbox{.85\textwidth}{\begin{center}
\texttt{#1}\end{center}} \rule {.25cm}{0cm}}\end{center}}
{}
}

\newtheorem{theorem}{Theorem}[section]   
\newtheorem{corollary}[theorem]{Corollary}
\newtheorem{lemma}[theorem]{Lemma}
\newtheorem{proposition}[theorem]{Proposition}

\theoremstyle{definition}
\newtheorem{definition}[theorem]{Definition}
\newtheorem{remark}[theorem]{Remark}

\allowdisplaybreaks
\numberwithin{equation}{section}


\begin{document}

\title[Existence of invariant measures]{Invariant measures for the one-dimensional stochastic Navier-Stokes-Korteweg equations}

\author[D. Donatelli, L. Pescatore, S. Spirito]{D. Donatelli, L. Pescatore, S. Spirito}
\address[D.Donatelli]{DISIM - Dipartimento di Ingegneria e Scienze dell'Informazione e Matematica\\ Universit\`a  degli Studi dell'Aquila \\Via Vetoio \\ 67100 L'Aquila \\ Italy}
\email[]{\href{mailto:donatella.donatelli@univaq.it}{donatella.donatelli@univaq.it}}
\address[L.Pescatore]{DISIM - Dipartimento di Ingegneria e Scienze dell'Informazione e Matematica\\ Universit\`a  degli Studi dell'Aquila \\Via Vetoio \\ 67100 L'Aquila \\ Italy}
\email[]{\href{mailto:lorenzo.pescatore@univaq.it} {lorenzo.pescatore@univaq.it}}
\address[S. Spirito]{DISIM - Dipartimento di Ingegneria e Scienze dell'Informazione e Matematica\\ Universit\`a  degli Studi dell'Aquila \\Via Vetoio \\ 67100 L'Aquila \\ Italy}
\email[]{\href{mailto:stefano.spirito@univaq.it}{stefano.spirito@univaq.it}}

\begin{abstract}
We investigate the long-time behaviour of a one-dimensional compressible viscous fluid with general capillarity and density dependent viscosity,  driven by a stochastic force.  In particular, we prove the existence of invariant measures by applying the Krylov-Bogoliubov method in a setting where the dynamics is supported on a non-complete phase space. This analysis is further enhanced by the derivation of a refined stability result determining the continuous dependence with respect to the initial data. The present paper exhibits some properties and results for Korteweg fluids which are not known in absence of the capillarity tensor.  In particular, we prove that the Markov semigroup associated with strong solutions is Feller and we can consider ranges of the adiabatic and viscosity exponents $\gamma$ and $\alpha$ larger than those available in the current ergodic literature for compressible fluids.  Also the interplay between the choice of the physical domain and the use of a damping term is discussed.
\end{abstract}

\maketitle
\section{Introduction}
\noindent
We consider a one-dimensional compressible viscous fluid with capillarity, endowed with periodic boundary conditions.  The dynamics is driven by a stochastic forcing term and it is described by the following set of equations:
\begin{equation}\label{main system}
\begin{cases}
\partial_t \rho+ \partial_x(\rho u)=0, \\ 
\text{d} (\rho u)+ [\partial_x (\rho u^2) + \partial_x p(\rho)]\text{d}t= [\partial_x \mathcal{S} + \partial_x \mathcal{K} -\varepsilon \rho u ]\text{d}t+ \rho \sigma(x) \text{d}W.
\end{cases}
\end{equation}
Here $\rho(x,t), u(x,t)$ denote the density and the velocity of the fluid respectively,  while regarding the constitutive law for the pressure $p$ we
consider the isentropic regime 
\begin{equation}
p(\rho)= a^2 \rho^\gamma, \quad \gamma \ge 1,
\end{equation}
with adiabatic exponent $\gamma$ and $a>0$ being a dimensionless parameter which is inversely proportional to the Mach number.
The viscous stress term $\mathcal{S}$ and the capillarity tensor $\mathcal{K}$ are of the form 
\begin{equation}
\mathcal{S}(\partial_x u)= \mu(\rho) \partial_x u, \quad \mathcal{K}(\rho, \partial_x \rho)= \rho \partial_x (k(\rho) \partial_x \rho) -\dfrac{1}{2} \big( k(\rho)+ \rho k'(\rho) \big) | \partial_x \rho |^2,
\end{equation}
for general viscosity and capillarity coefficients satisfying the power law relations
\begin{equation}\label{power laws cap visc}
\mu(\rho)=  \rho^\alpha, \quad \alpha \ge 0, \quad \; k(\rho)=  \rho^\beta, \quad \beta \in \mathbb{R}.
\end{equation}
Note that the capillarity term in the momentum equation can be rewritten as 
\begin{equation}\label{partial x K}
\partial_x \mathcal{K}(\rho, \partial_x \rho)= \rho \partial_x \bigg( \partial_x(k(\rho) \partial_x \rho)- \dfrac{k'(\rho)}{2} | \partial_x \rho |^2 \bigg).
\end{equation}
The Korteweg tensor $\mathcal{K}$ has been introduced in \cite{Korteweg} and subsequently it has been considered by Dunn and Serrin in the pioneering work \cite{Dunn and Serrin}.  We also refer to \cite{Gorban} for a kinetic description related to the non-truncated Chapman-Enskog expansion and to \cite{Anderson} for its use as a refined model for diffuse interface and phase transition. A review on these topics is collected in \cite{Pescatore}.
The choice of the capillarity exponent $\beta$ determines different structures of \eqref{partial x K} that can be found in fluid dynamics literature. In particular,  the case $\beta=-1$ for inviscid fluids corresponds to the Quantum-Hydrodynamic model \cite{Landau}, which in the multi-dimensional case reads $\operatorname{div} \mathcal{K}= 2\rho \nabla \big( \frac{ \Delta  \sqrt{\rho}}{\sqrt{\rho}} \big).$ On the other hand for $\beta=0$ the right hand side of \eqref{partial x K} reduces to $\rho \nabla \Delta \rho,$ that is commonly used as the basic model for capillary fluids. See also the discussion in \cite[Section~4]{Denzler}.
Our analysis also includes the contribution of a linear damping term $-\varepsilon \rho u,$ with $\varepsilon$ being a positive parameter, while the stochastic force $\rho \sigma(x) \text{d}W$ is a multiplicative noise which reduces to an additive one in the velocity formulation.  We refer to Section \ref{Sec2} for the precise definition of the stochastic setting.
\\
\\
System \eqref{main system} is endowed with the following deterministic initial conditions
\begin{equation}
\rho(x,0)= \rho_0(x), \quad u(x,0)=u_0(x), 
\end{equation}
having strictly positive initial density and normalized total mass. Throughout the paper, $\mathbb{T}$ denotes the one-dimensional flat torus with unit volume. 
The study of invariant measures is of fundamental importance to investigate the long-time behaviour of stochastically forced flows and to provide a connection between theoretical aspects and applications. From a probabilistic point of view, the ergodic theory focuses on the analysis of long-time properties of the fluid flow, investigating the existence and uniqueness of invariant measures, as well as the convergence of the dynamics toward statistical equilibrium in the long-time limit. To this purpose, a wide fluid dynamical literature concerning incompressible flows has been developed in recent years. In particular, we mention \cite{Flandoli 0},  \cite{Hairer},  \cite{Bricmont}, \cite{Flandoli-Maslowski},  \cite{Da Prato 2},  \cite{Kuksin} for the 2D case and \cite{Da Prato- Deb}, \cite{Flandoli-Romito}, \cite{Romito} in the 3D setting. For the latter, we also highlight the results \cite{Flandoli-Gatarek} where statistically stationary solutions are considered and \cite{Foias}, \cite{Vishik} for an analogue in a deterministic setting.  Nevertheless, the use of a damping term has been considered in \cite{Bessaih}, \cite{Brzezniak 3}, see also the recent results \cite{Brzezniak 2}, \cite{Brzezniak} for the analysis of invariant measures in the context of the 2D stochastic nonlinear damped Schr\"{o}dinger equations.
\\
\\
The scenario for compressible fluids is significantly different.  The problem is indeed strongly degenerate since the driving force is acting only on one component of the phase space and conserved quantities such as the total mass persist under its action.  Also, the ``hyperbolic'' structure of this wide class of fluid mechanics equations leads to a significant loss of well-posedness and in particular to the lack of uniqueness results for global in time solutions.  In the probabilistic setting, nonuniqueness prevents a canonical Markov semigroup from being obtained directly from the solution theory. Therefore a separate Markov selection or a trajectory-space formulation is usually required.
\\
\\
Our interest is to investigate the long-time behaviour of system \eqref{main system}. In particular, we focus on the existence of invariant measures and we provide regimes for the existence result in terms of the viscosity, capillarity and adiabatic exponents. The natural phase space for global solutions of \eqref{main system} is given by 
\begin{equation}\label{phase space}
\mathcal{X}= \bigg\{ (\rho,u) \in H^2(\mathbb{T}) \times H^1(\mathbb{T}) \; : \; \int_{\mathbb{T}} \rho(x)dx=1,  \; \rho >0 \bigg\},
\end{equation}
which is dictated by the well-posedness result Theorem \ref{Thm well posedness}.
\\
\\
The analysis of Korteweg viscous fluids strongly relies on a Bresch-Desjardins entropy approach which was first introduced in \cite{Bresch2} and \cite{BD} for the compressible Navier-Stokes equations, see also \cite{Mellet} and \cite{Const}  for the 1D case. A key role in this study is indeed played by the effective velocity $$V=u + \dfrac{\mu(\rho) \partial_x \rho}{\rho^2},$$ whose analysis determines the following conditions
\begin{equation}\tag{SCC}\label{SCC}
2 \alpha -4 \le \beta \le 2 \alpha -1
\end{equation}
and 
\begin{equation}\tag{NV}\label{NV}
0 \le \alpha \le \dfrac{1}{2} \quad \text{or} \quad \beta \le -2.
\end{equation}
In particular the \textit{strong coercivity condition} \eqref{SCC} introduced by Germain and LeFloch \cite{Le Floch} determines the admissibility range in which global solutions are expected. On the other hand condition \eqref{NV} prevents the formation of the vacuum states of the density, provided that they are not present at the initial time.  In order to investigate the existence of invariant measures one has to define a genuine Markovian framework.  This is ensured in the underlying regime,  where global existence and uniqueness of strong solutions is established, see Theorem \ref{Thm well posedness} and \cite{Pesc}.
The main outcome of the present paper is the following theorem
\begin{theorem}\label{main theorem}
Let $\alpha \ge 0,$ $\beta \in \mathbb{R}$ and $\gamma \ge 1.$  Assume that \eqref{noise H3} is satisfied and that one of the following conditions holds
\begin{itemize}
\item [(1)] $\alpha \in [0, \frac{1}{2}],\quad \beta \in [2\alpha-4,-1-\alpha],$
\item [(2)] $\alpha \in (\frac{1}{2}, 1], \; \; \, \, \beta \in [2\alpha-4,-2].$
\end{itemize}
Then, there exists a probability measure $\nu \in \mathfrak{B}(\mathcal{X}_{1,0})$ which is invariant under the Markov semigroup $\mathcal{P}_t$ generated by system \eqref{main system}.  Furthermore, the following inequality holds
\begin{equation}\label{ineq inv meas}
\int_{\mathcal{X}} \mathcal{D}[\rho,u] \,  d\nu(\rho,u) \le  \dfrac{1}{2}\| \sigma \|^2_{L^\infty}, 
\end{equation}
where  $\mathcal{D}[\rho,u]$ denotes the entropy dissipation functional given by \eqref{entropy dissipation} and $\mathcal{X}_{1,0}$ represents the phase space $\mathcal{X}$ endowed with the $ H^1 \times L^2$ metric.
\end{theorem}
\begin{remark}
In the zero-capillarity case, the global well-posedness conditions reduce to $\alpha \in [0,\frac{1}{2}],$ while the existence of invariant measures is obtained in \cite{Coti} by assuming constant viscosity and $\gamma=1.$ Our result covers several fluid dynamics models, including the Quantum-Navier-Stokes equations $\beta=-1,$ with constant viscosity and $\gamma \ge1$ and the shallow water system $\alpha=1,\gamma=2,$ augmented with the capillarity term (surface tension) given by \eqref{partial x K} with $k(\rho)=\rho^{-2}.$
\end{remark}
\subsection{Comparisons with previous results}
The investigation of the long-time behaviour of stochastically forced compressible fluids is very recent.  In the absence of uniqueness of global solutions, the analysis focuses on the study of statistically stationary solutions,  see \cite{Breit Feir Hof 5},\cite{Breit Feir Hof Mas} and also \cite{Fanelli3} for the deterministic case.  The common feature of these results is to consider the whole trajectory as an initial datum in the spirit of \cite{Foias} and \cite{Vishik}. The dynamics is then described by time shifts of solutions and different definitions of statistically stationary solutions can be considered.  Also convergence results of time averages are available, see \cite{Breit Feir Hof 5}, \cite{Fanelli3}.
To the best of our knowledge,  the only available result in the literature concerning the existence of invariant measures for compressible fluids is \cite{Coti} where the case of the one-dimensional stochastic Navier-Stokes equations is considered. For the latter model, we also mention \cite{G-Q Chen} for a vanishing Mach number limit and \cite{Gas} for the existence of an invariant measure for a discrete approximation of the system. Note, however, that in \cite{Coti} the restrictions $\gamma=1,\; \alpha=0$ occur and vanishing Dirichlet boundary conditions for the velocity and for the noise are assumed.  We therefore highlight the following elements:
\begin{itemize}
\item The presence of the capillarity term in system
\eqref{main system} allows us to deduce bounds for higher-order
derivatives of the density and to control the vacuum.
More precisely, under condition \eqref{T}, the entropy dissipation
functional $\mathcal{D}$ has compact sublevel sets in
$\mathcal{X}_{1,0}$. The expected time integral of $\mathcal{D}$
grows at most linearly in time. These estimates yield tightness,
and the restrictions $\gamma=1$ and $\alpha=0$ in \cite{Coti}
can therefore be relaxed by assuming \eqref{T}.
The ranges stated in Theorem \ref{main theorem} arise as the
intersection of the well-posedness conditions \eqref{SCC}
and \eqref{NV}, together with the following assumption
for tightness:
\begin{equation}\label{T}\tag{T}
\alpha+\beta+1\le0.
\end{equation}
\item 
By virtue of a refined stability result, Proposition \ref{continuous dependence on the initial data}, we prove that the Markov semigroup is Feller with respect to the $H^1 \times L^2$ metric. This is in contrast to the case without capillarity, where an additional entropy constraint is required, see \cite[Theorem 2.7]{Coti}. The proof of the stability result relies on a relative entropy approach which is new in the considered regimes of the $\alpha$ and $\beta.$
Furthermore,  Theorem \ref{main theorem} provides no restrictions on $\gamma$ and it also accounts for the case of non-constant viscosity.  
\\
\item The extension from the interval $[0,1]$ to the one-dimensional flat torus $\mathbb{T}$ gives rise to some issues concerning the energy growth and the choice of the velocity component of the phase space $\mathcal{X}.$ The stochastic forcing term is indeed adding energy to the system and no vanishing boundary conditions are considered.  We overcome these problems by introducing a linear damping term $-\varepsilon \rho u,$ which is of typical interest in the analysis of energetically open systems, see \cite{Kuan}, \cite{Bessaih}, \cite{Brzezniak 3}, \cite{Brzezniak 2}, \cite{Brzezniak}.  To be precise, it allows us to perform a stochastic compactness argument, see Proposition \ref{Prop tight}, by recovering an appropriate control of $\| u \|_{L_x^2}$ in terms of the energy and density norms. 
\end{itemize}
\noindent
The existence of invariant measures is proved by using the Krylov-Bogoliubov method. In particular, we define the collection of the so-called time-averaged measures $\{\nu_T\}_{T >0},$ and then we perform a stochastic compactness argument by virtue of a Bresch-Desjardins entropy approach. Note that despite this being a quite standard procedure, several difficulties arise.  The nonlinear structure of the Korteweg tensor $\mathcal{K},$ with general exponent $\beta \in \mathbb{R},$ is indeed non-trivial to handle in the a priori estimates. Furthermore, the phase space $\mathcal{X}$ is not complete due to the condition $\rho > 0,$ therefore the tightness argument must show that a limiting probability measure in the ambient space is supported on $\mathcal{X}$.  To show that $\nu$ is also an invariant measure, one has to exploit the regularity properties of the Markov semigroup $\mathcal{P}_t.$ To this purpose we highlight that the natural $H^2 \times H^1$ topology on $\mathcal{X}$ does not reflect the available continuous dependence with respect to the initial data.  This issue arises from the absence of uniqueness results within the class of finite energy/entropy solutions which is typical in the scenario of compressible fluids. We overcome this problem by providing an $H^1 \times L^2$ based continuity result which relies on a relative entropy approach and allows us to establish the Feller property. The stability result for $\beta \in \mathbb{R}$ is novel and it strongly relies on the skew-adjoint structure of \eqref{main system}, which is revealed through a suitable change of variable, see \cite{Benz}.  We also highlight that the relative entropy analysis covers the case of general exponents $\alpha$ and $\beta.$ Therefore we believe that this result is of independent interest and may be employed across different contexts.
\subsection{Open problems}
This result gives rise to further developments concerning uniqueness
and the analysis of singular limits. Concerning the uniqueness
of invariant measures, we note that the unforced system converges
exponentially to the steady state $(1,0)$, which is selected as the unique steady state by the presence of the damping.
On the other hand, the dissipation inequality \eqref{ineq inv meas}
suggests that in the low Mach number limit $a\rightarrow\infty$
the measure $\nu=\nu_a$ concentrates on the set
$\{\partial_x\rho=0\}$, and thus $\rho=1$ because of the mass
constraint. Note that, in contrast to the zero-Dirichlet setting discussed in \cite[Remark 1.2]{Coti}, the formally projected limit dynamics on the torus is consistent with the stochastic forcing without any rescaling of $\sigma$ with respect to the Mach number. This dynamics is described by
\begin{equation}\label{Orn-Ulh}
\begin{split}
\text{d}u
&=-\varepsilon u\,\text{d}t
+\bigg(\int_{\mathbb{T}}\sigma(x)\,dx\bigg)\text{d}W,
\qquad \partial_xu=0,
\end{split}
\end{equation}
which is an Ornstein--Uhlenbeck process admitting a unique
invariant measure.
The identification of the low Mach number limit of $\nu_a$
also requires an analysis of the acoustic fluctuations, see \cite{G-Q Chen}. We aim to address the uniqueness problem for the invariant
measure of system \eqref{main system}, together with the analysis
of this singular limit, in future research.
\subsection{State of the art for deterministic and stochastic Korteweg fluids}
In the one-dimensional case, the analysis of the deterministic Navier-Stokes-Korteweg equations has been addressed in \cite{Le Floch} and global existence of finite energy weak solutions has been proved.  Recently, Antonelli, Bresch and Spirito in \cite{Ant} prove global existence of weak solutions in the range $2\alpha-3 \le \beta < 2\alpha-1,$ which corresponds to the \textit{tame capillarity condition} introduced in \cite{Le Floch} in which viscous effects are dominant over capillarity in the dynamics.
We also refer to \cite{Chen} and \cite{Burtea1}, where the authors considered capillarity and viscosity coefficients coupled through appropriate relations.  Concerning the global well-posedness in the framework of strong solutions in the regime given by \eqref{SCC} and \eqref{NV} we refer to \cite{Pesc} and to \cite{D.P.S.} for the case of the Quantum-Navier-Stokes equations. Furthermore, the latter model has been also studied for solutions in which the vacuum is allowed and global existence of weak solutions is obtained for $ \alpha \in (\frac{1}{2}, 1].$ Note that although the results \cite{Pesc}, \cite{D.P.S.} and \cite{D.P.S.2} account for the case of stochastically forced flows, they also cover the deterministic case.
To the best of our knowledge, the case $d=2,3$ with general capillarity and viscosity satisfying \eqref{power laws cap visc} is an open problem and no results of global solutions for large data are available with the current methods. In the case of the Quantum-Navier-Stokes equations with degenerate linear viscosity, global existence of weak solutions has been proved in \cite{Spirito2}, \cite{Lacroix},  \cite{Jung qns}, \cite{Vasseur Yu 1}, \cite{Ant6}, \cite{Jung qns 2}.
The case of constant capillarity, $\beta=0,$ has been addressed in \cite{Hatt} and \cite{Spirito} where local well-posedness for strong solutions and global existence of weak solutions are obtained. We also mention \cite{Fanelli1} and \cite{Fanelli2} where the case of rotating fluids with capillarity is considered.
For the inviscid case with general capillarity coefficients, namely the Euler-Korteweg equations, we refer to \cite{Benz} for the local existence of smooth solutions and to \cite{Aud} for small data well-posedness.  Also a non-uniqueness result using convex integration has been obtained in \cite{Feir. Don.}. Furthermore, the Quantum-Hydrodynamics system has been widely studied in the series of papers \cite{Ant2}, \cite{Ant1}, \cite{Ant3}, \cite{Ant4} in the framework of global finite energy weak solutions.
\\
\\
The analysis of stochastic Korteweg fluids is recent and it consists of the series of aforementioned results \cite{D.P.S.}, \cite{D.P.S.2}, \cite{Pesc}. In contrast, most of the literature deals with the compressible Navier-Stokes equations with constant viscosity, for which well-posedness and related results have been obtained in \cite{Breit Hofmanova} and subsequently in the series of works by Breit, Feireisl and Hofmanov\'{a} \cite{Breit Feir Hof 1}, \cite{Breit Feir Hof 2}, \cite{Breit Feir Hof 3}, \cite{Breit Feir Hof 5} and their monograph \cite{Feir}; see also \cite{Breit Feir Hof 4} for the inviscid case. A recent analysis concerning the study of stochastic perturbations of transport type can be found in \cite{Breit Feir Hof Zat} and \cite{Breit Feir Hof Much}. Finally, in \cite{Zatorska} the authors proved sequential stability of weak martingale solutions for the compressible Navier-Stokes equations with degenerate linear viscosity and some ill-posedness results for the stochastic full Euler system have been obtained in \cite{Chiod Feir Fland} by means of stochastic convex integration methods. Further results on stochastic compressible fluids can be found in \cite{Breit Feir Hof 2}, \cite{Breit Feir Hof 5}, \cite{Breit Feir Hof Mas}, \cite{Coti} and the recent contribution \cite{Kuan}, where the stochastic isentropic Euler equations with linear damping are considered.
\\
\\
\textbf{Organization of the paper.} The paper is organized as follows. We devote Section \ref{Sec2} to the description of the stochastic setting and to the well-posedness results for system \eqref{main system}. We also introduce the Markovian framework and we present some structural properties of the Markov semigroup associated with compressible fluid flows. In Section \ref{Sec3} we provide a BD entropy method which allows us to prove suitable a priori estimates.
Section \ref{Sec4} is devoted to the relative entropy approach determining a uniqueness result and the appropriate continuous dependence on the initial data.  The regularity of the Markovian framework and the Krylov-Bogoliubov method are then considered in Section \ref{Sec5},  together with the proof of our main result, Theorem \ref{main theorem}. To conclude, Appendix \ref{Appendix} is devoted to the proof of the relative entropy balance \eqref{eq rel entropy }, while Appendix \ref{Appendix B} contains some exponential moment bounds that we aim to use for future research.
\\
\\
\textbf{Acknowledgements.} The authors are grateful to Michele Coti Zelati and Lucio Galeati for fruitful discussions. 
The authors gratefully acknowledge the partial support by the Gruppo
Na\-zio\-na\-le per l’Analisi Matematica, la Probabilit\`a e le loro
Applicazioni (GNAMPA) of the Istituto Nazionale di Alta Matematica
(INdAM), and by the PRIN 2020 ``Nonlinear evolution PDEs, fluid
dynamics and transport equations: theoretical foundations and
applications'' and by the PRIN2022
``Classical equations of compressible fluids mechanics: existence and
properties of non-classical solutions''. The first and the third authors gratefully acknowledge the partial support by PRIN2022-PNRR ``Some mathematical approaches to climate change and its impacts.''
\section{Stochastic setting and well-posedness}\label{Sec2}
\noindent
In this Section we recall the well-posedness results for system \eqref{main system} and we introduce the Markovian framework associated with strong solutions. We also provide a description of the Markovian properties of one-dimensional compressible fluid flows.
\\
\\
\textbf{Notation and conventions.} Besides the standard notations used in the literature, we consider the following conventions.  The Lebesgue space is denoted by $L^p(\mathbb{T})$ with related norm $\| \cdot \|_{L^p}.$ Similarly $W^{k,p}(\mathbb{T})$ denotes the Sobolev space of $L^p(\mathbb{T})$ functions with $k$ distributional derivatives in $L^p(\mathbb{T})$ and $H^k(\mathbb{T})$ corresponds to the case $p=2.$ For a Banach space $X$, we denote by $C([0,T];X)$
the space of continuous $X$-valued functions and by
$L^p(0,T;X)$ the usual Bochner spaces.
The functional spaces $\mathcal{M}_b(\mathcal{X})$ and $C_b(\mathcal{X})$ denote the set of real-valued bounded measurable functions and real-valued continuous and bounded functions respectively. $\mathcal{B}(\mathcal{X})$ represents the family of Borel subsets of $\mathcal{X},$ while $\mathfrak{B}(\mathcal{X})$ is the set of probability measures on $\mathcal{X}.$ For simplicity, we normalize the mass to one and we assume $|\mathbb{T}|=1.$
\subsection{Stochastic setting}
We consider  a stochastic basis $(\Omega,\mathcal{F},(\mathcal{F}_t)_{t\ge0},\mathbb{P})$ with complete, right-continuous filtration and mutually independent standard $(\mathcal{F}_t)$-Wiener processes $W^k$. The forcing term in \eqref{main system} is
\begin{equation}\label{stoch force}\begin{split}
\sigma\text{d}W=\sum_{k=1}^{\infty}\sigma_k(x)\text{d}W^k(t),
\end{split}\end{equation}
where the deterministic coefficients satisfy
\begin{equation}\label{noise H3}\begin{split}
\sum_{k\ge1}\|\sigma_k\|_{H^2(\mathbb T)}^2<\infty.
\end{split}\end{equation}
In particular, this hypothesis controls the spatial means as well as
the derivatives, and
\begin{equation}\begin{split}
\|\sigma\|_{L^\infty}^2
&:=\sup_{x\in\mathbb T}\sum_{k\ge1}|\sigma_k(x)|^2<\infty,\\
\|\partial_x\sigma\|_{L^\infty}^2
&:=\sup_{x\in\mathbb T}\sum_{k\ge1}|\partial_x\sigma_k(x)|^2<\infty.
\end{split}\end{equation}
The stochastic integral is understood in the It\^{o} sense.
The series $w(t)=\sum_k\sigma_kW^k(t)$ has almost surely continuous
paths in $H^2(\mathbb T)$. Thus the change of variables
$\tilde u=u-w$ yields the random PDE \eqref{random pde}, to which
deterministic pathwise estimates apply.
We use the following global well-posedness result, which follows from the above change of variables, the a priori estimates of Section \ref{Sec3} and appropriate deterministic arguments. See also \cite{Pesc} for a related
well-posedness result in a multiplicative noise setting.
\begin{theorem}\label{Thm well posedness}
Let $(\Omega,\mathcal{F},(\mathcal{F}_t)_{t\ge0},\mathbb{P})$ be 
a stochastic basis with right continuous filtration. Let $(\rho_0,u_0)\in\mathcal{X}$ and assume that \eqref{noise H3} holds. Then there exists a unique global strong pathwise solution $(\rho(t;\rho_0),u(t;u_0))$ to system \eqref{main system} in the following regularity class:
\begin{equation}
\begin{split}
&\rho\in C([0,T];H^2(\mathbb{T}))\cap L^2(0,T;H^3(\mathbb{T})),\\
&u\in C([0,T];H^1(\mathbb{T}))\cap L^2(0,T;H^2(\mathbb{T})),\\
& \rho \in L^{\infty}((0,T) \times \mathbb{T}), \quad \dfrac{1}{\rho} \in L^{\infty}((0,T) \times \mathbb{T}),
\end{split}
\end{equation}
$\mathbb{P}$- almost surely, provided that \eqref{SCC} and \eqref{NV} are satisfied.
\end{theorem}
\begin{remark}
For completeness, we highlight the main elements of the proof of Theorem \ref{Thm well posedness}. We first carry out the construction for regular initial data and smooth spatial noise coefficients.
\\
\begin{itemize}
\item [(1)] Local well-posedness of strong pathwise solutions is established for any $\gamma \ge 1, \, \alpha \ge 0, \, \beta \in \mathbb{R}$ by using standard approximating schemes which appropriately truncate the non-linearities. A blow-up alternative then occurs in terms of both $\| u \|_{W_x^{2,\infty}}$ and $\| \rho \|_{W_x^{2,\infty}}.$
\\
\item [(2)] Condition \eqref{SCC} characterizes the range in which the quantity $-\int_{\mathbb{T}} \partial_x \mathcal{K} \cdot \mu(\rho) \frac{\partial_x \rho}{\rho^2} dx$ is non-negative. This property is crucial for the BD entropy estimates in Proposition \ref{Prop entropy ineq Q/2} and Proposition \ref{Prop path BD entropy}, which yield bounds for $\| \partial_x \rho^{\alpha-\frac{1}{2}} \|_{L_t^\infty L_x^2}$ and $\| \partial_{xx} \rho^{\frac{\alpha+\beta+1}{2}} \|_{L_{t,x}^2}$. In particular, the former bound and the capillary energy prevent the formation of vacuum under \eqref{NV}, while the latter estimate provides the starting point for a bootstrap argument yielding an $L_t^\infty H_x^s$ estimate for $s>3$. Thus, the unique local-in-time solution can be extended to a global one, provided that \eqref{SCC} and \eqref{NV} hold.
\\
\item [(3)] The regularity class of global strong solutions can be reduced to that stated in Theorem \ref{Thm well posedness} by combining mollification of the initial data and the spatial noise coefficients with uniform a priori estimates, stochastic compactness
and pathwise uniqueness.
\end{itemize}
\end{remark}
\subsection{Markovian properties of compressible fluids}
We introduce the Markovian framework and we recall some basic definitions that will be used in the sequel. We also provide a description of structural properties of compressible fluid models which determine a non-trivial scenario for the analysis of the Markovian setting. We start with the introduction of the phase space $\mathcal{X}$ and with the definitions of transition functions and Markov semigroup for system \eqref{main system}.
\\
\\
By virtue of the regularity class of solutions given by Theorem \ref{Thm well posedness}, we consider the following phase space.
\begin{equation}
\mathcal{X}= \bigg\{ (\rho,u) \in H^2(\mathbb{T}) \times H^1(\mathbb{T}) \; : \; \int_{\mathbb{T}} \rho(x)dx=1,  \; \rho >0 \bigg\}.
\end{equation}
Note however that,  due to the open condition $\rho>0,$ the phase space $\mathcal{X}$ is not complete when endowed with the natural $H^2(\mathbb{T}) \times H^1(\mathbb{T})$ metric $( \mathcal{X}_{2,1}).$ On the other hand the appropriate topology which is reflecting the continuous dependence with respect to the initial data turns out to be the one induced by the $H^1(\mathbb{T}) \times L^2(\mathbb{T})$  metric, namely $\mathcal{X}_{1,0},$ which is still not complete. This will be pointed out in Proposition \ref{continuous dependence on the initial data} and it is related to some structural properties of general deterministic and stochastic compressible fluid equations which we highlight below.
\begin{remark}\label{remark structural prop compressible}
\leavevmode
\begin{enumerate}[label=\arabic*), labelindent=0pt]
\item 
It is not known, in general, whether solutions within the class of finite energy/entropy are unique. In order to get uniqueness one usually considers an initial datum $u_0 \in H^1$ and, assuming also that the initial density $\rho_0$ is sufficiently regular and bounded away from zero, global existence and uniqueness of strong solutions can be obtained. This regularity reflects the structure of the phase space for compressible fluid flows.
\\
\item 
In the case of Korteweg fluids, the momentum equation involves three spatial derivatives on the density. Since the continuity equation yields loss of  derivatives with respect to the velocity $u,$ then the derivation of an $H^1$ estimate for the velocity $u$ is non-trivial and it is obtained by virtue of a suitable change of variables which determines a skew-adjoint structure of \eqref{main system}. This method provides several fundamental cancellations between high-order derivative terms and, as highlighted in the proof of Proposition \ref{Prop path HS est}, it couples the $H^1(\mathbb{T})$ estimate on $u$ with an $H^2(\mathbb{T})$ estimate for $\rho.$ 
\\
\item 
The continuous dependence with respect to the initial data is usually obtained by means of relative energy/entropy functionals, together with an appropriate control on the vacuum set. These functionals are defined through first order Taylor expansions of the kinetic and internal energy (see also \cite{Bresch0} for a different approach). Due to the structure of the energy, these functionals exhibit a mismatch in derivative order with respect to the considered phase space. In particular, as exploited in \eqref{def energy}, we note that for Korteweg fluids the energy $H(\rho,u)$ involves the $H^1$ regularity of $\rho$ and $\|\sqrt{\rho}u\|_{L^2}$.
\end{enumerate}
\end{remark}
For these reasons, although the definitions stated below can be formulated for a generic metric space $(\mathcal{X},d),$ we present them directly with respect to $\mathcal{X}_{1,0}.$ It is important to mention that since $\mathcal{X}_{2,1}$ is an open subset of a Polish space, then it is also Polish. Therefore since the inclusion $i: H^2 (\mathbb{T})\times H^1(\mathbb{T})  \hookrightarrow H^1 (\mathbb{T})\times L^2 (\mathbb{T})$ is continuous and injective, then the two topologies generate the same Borel $\sigma-$ algebra on the space $\mathcal{X}.$ Moreover, the test functions in the definition of invariant measure, Definition \ref{definition invariant measure}, can be taken in $C_b(\mathcal{X}_{1,0})$ instead of $\mathcal{M}_{b}(\mathcal{X}_{1,0})$. We refer the reader to \cite[Remark 5.1]{G-Q Chen},  for a similar scenario.
\begin{definition}
Let $(\rho(t;\rho_0),u(t;u_0))$ be the unique strong solution of \eqref{main system} with deterministic initial data $(\rho_0,u_0).$ The transition functions $\mathcal{P}_t$ are defined by 
\begin{equation*}
\mathcal{P}_t(\rho_0, u_0, B)= \mathbb{P}((\rho(t; \rho_0), u(t;u_0)) \in B),
\end{equation*}
with related Markov semigroup
\begin{equation*}
\mathcal{P}_t \phi(\rho_0,u_0)= \mathbb{E} \phi( \rho(t;\rho_0), u(t;u_0)), \qquad \phi \in \mathcal{M}_b(\mathcal{X}_{1,0}).
\end{equation*}
\end{definition}
\noindent
Note that by virtue of the well-posedness result given by Theorem \ref{Thm well posedness}, then $\mathcal{P}_t$ is well defined and it satisfies the usual semigroup properties 
$$ \mathcal{P}_0= \mathbb{I}_{\mathcal{M}_b(\mathcal{X}_{1,0})}, \qquad \mathcal{P}_{t+s}=\mathcal{P}_t \circ \mathcal{P}_s.$$
Our goal is to prove the existence of an invariant probability measure. The precise definition is stated as follows:
\begin{definition}\label{definition invariant measure}
Let $\mathcal{P}_t$ be the Markov semigroup associated to system \eqref{main system}. A measure $\nu \in \mathfrak{B}(\mathcal{X}_{1,0})$ is said to be invariant under $\mathcal{P}_t$ if 
\begin{equation}\label{def inv meas}
\int_{\mathcal{X}} \mathcal{P}_t \phi \, d \nu = \int_{\mathcal{X}} \phi\, d \nu, \quad \forall \phi \in C_b(\mathcal{X}_{1,0}),
\end{equation}
for all $t \ge 0.$
\end{definition}
\noindent
A fundamental property for the study of the existence of invariant measures is the following:
\begin{definition}(Feller property)
\\
The Markov semigroup $\mathcal{P}_t$ associated to system \eqref{main system} is said to be \textit{Feller} if the following property holds
\begin{equation}
 \mathcal{P}_t: C_b( \mathcal{X}_{1,0}) \longrightarrow C_b( \mathcal{X}_{1,0}),
\end{equation}
for all $t \ge 0.$
\end{definition}
\subsection{Comparison with the zero-capillarity case}
In \cite{Coti}, the authors investigated the existence of invariant measures for the one-dimensional compressible Navier-Stokes equation in the unit interval $(0,1)$ with zero Dirichlet boundary condition. We first note the extension to the one-dimensional flat torus can be obtained by introducing the damping term $-\varepsilon \rho u$ in the momentum equation. In \cite{Coti} the following phase space is considered:
\begin{equation*}
\mathcal{X}^{NS}= \bigg\{ (\rho,u) \in H^1(0,1) \times H_0^1(0,1) \; : \; \int_{0}^{1} \rho(x)dx=1,  \; \rho >0 \bigg\},
\end{equation*}
and the analysis is carried over to $\mathcal{X}^{NS}$ endowed with the topology induced by the $L^2 \times L^2$ metric, which we denote by $\mathcal{X}^{NS}_{0,0}.$ Note that in this case the energy $H(\rho,u)= \int_{0}^{1} \frac{1}{2} \rho u^2+ a^2\rho \log \rho \, dx$ does not control the gradient of the density. 
On the one hand, the space $\mathcal{X}^{NS}_{0,0}$ is the appropriate metric space to apply the stochastic compactness argument. On the other hand the proof of the Feller property $\mathcal{P}_t: C_b( \mathcal{X}^{NS}_{0,0}) \longrightarrow C_b( \mathcal{X}^{NS}_{0,0}),$ seems to fail because of the vacuum set.  To overcome this problem in \cite{Coti} the following class is introduced:
\begin{definition}\label{def G}
A scalar measurable function $\phi \in \mathcal{M}_b(\mathcal{X}^{NS}_{0,0})$ belongs to the class $\mathcal{G}$ if the following convergence result holds:
\begin{equation}
\lim_{n \rightarrow \infty} \phi(\rho^n, u^n)= \phi(\rho,u),
\end{equation}
provided that 
\begin{equation}\label{hp G}
\{ (\rho^n,u^n) \}_{n \in \mathbb{N}} \subset \mathcal{X}^{NS}_{0,0} \quad \lim_{n \rightarrow \infty} (\rho^n, u^n)= (\rho,u) \; \text{in} \; \mathcal{X}^{NS}_{0,0}, \quad \sup_{n \in \mathbb{N}} \mathcal{E} (\rho^n, u^n) < \infty.
\end{equation}
\end{definition}
\noindent
The entropy functional $\mathcal{E}(\rho,u)$ is the analogue of \eqref{BD Def} in the case of constant viscosity, linear pressure and zero capillarity. On the one hand, we have $C_b(\mathcal{X}^{NS}_{0,0}) \subset \mathcal{G} \subset C_b(\mathcal{X}^{NS}_{1,1})$. The class $\mathcal{G}$ is invariant under the Markov semigroup $\mathcal{P}_t$ and the convergence of the time-averaged measures extends to test functions in $\mathcal{G}$; see \cite[Lemmas 3.2 and 3.6]{Coti}. On the other hand, invariance of $\mathcal{G}$ only yields continuity under the additional entropy bound in \eqref{hp G}; it does not by itself establish the Feller property on $\mathcal{X}^{NS}_{0,0}$. The class $\mathcal{G}$ coincides with $C_b(\mathcal{X}^{NS},\tau)$, where $\tau$ is the inductive limit topology of the entropy sublevels, each endowed with the $L^2\times L^2$ topology. In the present setting, $\tau$ is not metrizable.
\\
In the case of Korteweg fluids, the capillarity term provides additional control of the density gradient. In particular, the relative energy functional involves the gradient of the density and the convergence $(\rho^n,u^n)\to(\rho,u)$ in $\mathcal{X}_{1,0}$ yields $\rho^n\to\rho$ uniformly. Positivity of the limit then gives a common lower density bound, and the $H^1\times L^2$ convergence yields a uniform bound on $\mathcal{E}(\rho^n,u^n)$. However, the proof of the stability result in Proposition \ref{continuous dependence on the initial data} is significantly more involved for capillary fluids, since system \eqref{main system} involves the nonlinear structure of the Korteweg tensor with general exponent $\beta \in \mathbb{R}$ and density-dependent viscosity with $\alpha\ge0$.
\section{A priori estimates}\label{Sec3}
\noindent
This section is devoted to the derivation of the a priori estimates that will be used in order to prove the main result of this paper. In particular, we provide the energy and the BD entropy inequalities.  Furthermore, we use an appropriate change of variable which allows us to reduce system \eqref{main system} to a random PDE providing pathwise estimates. 
\\
\\
For a given smooth solution $(\rho,u)$ of \eqref{main system} we define the following energy functional 
\begin{equation}\label{def energy}
H(\rho,u)= \int_{\mathbb{T}} \bigg( \dfrac{1}{2} \rho u^2 + h(\rho)+ \dfrac{k(\rho)}{2}| \partial_x \rho |^2 \bigg) dx
\end{equation}
with 
\begin{equation*}
h(\rho)= \begin{cases} a^2 \dfrac{\rho^\gamma}{\gamma-1}, \quad \gamma> 1, \\
a^2 \rho \log \rho, \quad \gamma=1.
\end{cases} 
\end{equation*}
Then we introduce the following quantity which plays a fundamental role in our analysis
\begin{equation}\label{BD Def}
\begin{split}
\mathcal E(\rho,u)=H(\rho,u)+\frac12\int_{\mathbb T}\left(\mu(\rho)\frac{\partial_x\rho\,u}{\rho}+\mu^2(\rho)\frac{|\partial_x\rho|^2}{2\rho^3}\right)dx+\int_{\mathbb T}I(\rho;\alpha;\varepsilon)\,dx,
\end{split}
\end{equation}
where the damping potential is normalized at $\rho=1$:
\begin{equation*}
\begin{split}
I(\rho;\alpha;\varepsilon)&=\frac{\varepsilon}{2\alpha(\alpha-1)}\big[\rho^\alpha-1-\alpha(\rho-1)\big],\qquad \alpha\ne0,1,\\
I(\rho;1;\varepsilon)&=\frac\varepsilon2(\rho\log\rho-\rho+1),\qquad
I(\rho;0;\varepsilon)=\frac\varepsilon2(\rho-1-\log\rho).
\end{split}\end{equation*}
By conservation of mass, adding or subtracting constants and terms linear in $\rho$ does not affect the entropy balance.
These terms are chosen so that $I(1)=I'(1)=0$, which ensures that $I(\rho)$ is nonnegative since $I''(\rho)=\frac{\varepsilon}{2}\rho^{\alpha-2}>0$.
\\
We highlight that $\mathcal{E}(\rho,u)$ is related to the effective velocity $V=u+\mu(\rho) \frac{\partial_x \rho}{\rho^2}.$ In particular, the second term in the right hand side of \eqref{BD Def} can be expressed in terms of $$Q:= \mu(\rho) \frac{\partial_x \rho}{\rho^2}$$ as follows
\begin{equation}\label{entropy with Q}
\mathcal{E}(\rho,u)=\frac{1}{2}H(\rho,u) +\frac{1}{2}\int_{\mathbb{T}} \bigg( \dfrac{1}{2} \rho \big| u+ Q \big|^2+ h(\rho)+ \dfrac{k(\rho)}{2}| \partial_x \rho |^2+2 I(\rho \, ; \alpha \, ; \varepsilon) \bigg) dx,
\end{equation} 
and we observe that the quantity 
\begin{equation*}
\int_{\mathbb{T}} \bigg( \dfrac{1}{2} \rho \big| u+ Q \big|^2+ h(\rho)+ \dfrac{k(\rho)}{2}| \partial_x \rho |^2 \bigg) dx
\end{equation*}
is exactly the BD Entropy used in several contributions \cite{D.P.S.}, \cite{D.P.S.2}, \cite{Mellet}, \cite{Const}, \cite{Pesc}.
The decomposition \eqref{entropy with Q} will be very useful in the derivation of the BD entropy estimate, Proposition \ref{Prop entropy ineq Q/2}. In particular, $\mathcal E\ge0$ and we highlight the following properties
\begin{equation}\label{entropy elementary controls}
\begin{split}
&\|\sqrt\rho\,u\|_{L^2}^2+\|\sqrt\rho\,Q\|_{L^2}^2+\int_{\mathbb T}\rho^\beta|\partial_x\rho|^2dx+\int_{\mathbb T}I(\rho;\alpha;\varepsilon)\,dx\le C\mathcal E(\rho,u),\\
&\int_{\mathbb T}\big(\rho^\alpha+p(\rho)\big)dx\le C\big(1+\mathcal E(\rho,u)\big),
\end{split}
\end{equation}
which can be obtained with direct computations, for a generic positive constant $C$ depending on parameters $\alpha,\beta,\gamma,a,\varepsilon.$
We start our analysis with the following energy inequality which is obtained by using standard It\^{o} calculus.
\begin{proposition}\label{Prop energy ineq}
Let $(\rho,u)$ be a solution of \eqref{main system} with initial condition $(\rho_0,u_0) \in \mathcal{X}.$ Then the following energy inequality holds
\begin{equation}\label{energy inequ}
\mathbb{E} H(\rho,u)(t) + \mathbb{E} \int_{0}^{t} \| \rho^{\frac{\alpha}{2}} \partial_x u \|^2_{L^2} ds +\mathbb{E} \int_{0}^{t} \varepsilon \| \sqrt{\rho}u \|^2_{L^2}ds \le H(\rho_0,u_0) +\dfrac{1}{2} \| \sigma \|^2_{L^\infty} t, \quad \forall t\ge0.
\end{equation}
\end{proposition}
\begin{proof}
It\^o's formula and integration by parts give
\begin{equation*}\begin{split}
&\mathrm d\left(\frac12\int_{\mathbb T}\rho u^2dx\right)+\left(\|\rho^{\frac{\alpha}{2}}\partial_xu\|_{L^2}^2+\varepsilon\|\sqrt\rho\,u\|_{L^2}^2\right)\mathrm dt\\
&=\left(-\int_{\mathbb T}\partial_xp(\rho)u\,dx+\int_{\mathbb T}\partial_x\mathcal K\,u\,dx\right)\mathrm dt
+\sum_{l\ge1}\left(\int_{\mathbb T}\rho u\sigma_l\,dx\right)\mathrm dW^l
+\frac12\sum_{l\ge1}\int_{\mathbb T}\rho|\sigma_l|^2dx\,\mathrm dt.
\end{split}
\end{equation*}
The continuity equation implies
\begin{equation*}
\begin{split}
\int_{\mathbb T}\partial_xp(\rho)u\,dx&=\frac{d}{dt}\int_{\mathbb T}h(\rho)\,dx,\\
\int_{\mathbb T}\partial_x\mathcal K\,u\,dx&=\int_{\mathbb T}\left(\partial_x(k(\rho)\partial_x\rho)-\frac{k'(\rho)}2|\partial_x\rho|^2\right)\partial_t\rho\,dx
=-\frac{d}{dt}\int_{\mathbb T}\frac{k(\rho)}2|\partial_x\rho|^2dx.
\end{split}
\end{equation*}
Consequently,
\begin{equation}\label{energy equality}
\begin{split}
\mathrm dH(\rho,u)+\left(\|\rho^{\frac{\alpha}{2}}\partial_xu\|_{L^2}^2+\varepsilon\|\sqrt\rho\,u\|_{L^2}^2\right)\mathrm dt
=\sum_{l\ge1}\left(\int_{\mathbb T}\rho u\sigma_l\,dx\right)\mathrm dW^l
+\frac12\sum_{l\ge1}\int_{\mathbb T}\rho|\sigma_l|^2dx\,\mathrm dt.
\end{split}
\end{equation}
Finally, we integrate in time and after having applied the expectation we observe that the first term in the right-hand side of \eqref{energy equality} vanishes due to the martingale property and 
\begin{equation*}
\sum_{l=1}^{\infty} \int_{0}^{t} \int_{\mathbb{T}} \rho |\sigma_l|^2 dxds \le \| \sigma \|^2_{L^\infty} t,
\end{equation*}
hence we deduce \eqref{energy inequ}.
\end{proof}
\noindent
In the next Proposition we provide the BD entropy estimate.
\begin{proposition}\label{Prop entropy ineq Q/2}
Let $(\rho,u)$ be a solution of \eqref{main system} with initial condition $(\rho_0,u_0) \in \mathcal{X}.$ Assume that \eqref{SCC} is satisfied, then the following entropy inequality holds
\begin{equation}\label{entropy ineq Q/2}
\begin{split}
& \mathbb{E} \mathcal{E}(\rho,u)(t)+  \mathbb{E}\int_{0}^{t} \mathcal{D}[\rho,u] ds  \le \mathcal{E}(\rho_0,u_0) +\dfrac{1}{2} \| \sigma \|^2_{L^\infty} t, \quad 	\forall t\ge0,
\end{split}
\end{equation}
where the entropy dissipation functional $\mathcal D[\rho,u]$ is defined by:
\begin{itemize}
\item[(i)] For $\alpha+\beta+1\ne0$,
\begin{equation}\label{entropy dissipation}
\begin{split}
\mathcal D[\rho,u]={}&
\frac{\gamma a^2}{2}\int_{\mathbb T}\rho^{\gamma+\alpha-3}|\partial_x\rho|^2\,dx
+\frac12\int_{\mathbb T}\rho^\alpha|\partial_xu|^2\,dx
+\varepsilon\int_{\mathbb T}\rho|u|^2\,dx\\
&+\frac{2}{(\alpha+\beta+1)^2}
\int_{\mathbb T}\left|\partial_{xx}\rho^{\frac{\alpha+\beta+1}{2}}\right|^2\,dx
+\frac{c(\alpha,\beta)}2
\int_{\mathbb T}\left|\partial_x\rho^{\frac{\alpha+\beta+1}{4}}\right|^4\,dx,
\end{split}
\end{equation}
where
\begin{equation}\label{def c(alpha,beta)}
c(\alpha,\beta)=\frac{64}{(\alpha+\beta+1)^2}
\left[
\frac{(\alpha-\beta-1)(1-\alpha)}{(\alpha+\beta+1)^2}
-\frac{\beta}{3(\alpha+\beta+1)}
\right].
\end{equation}
\item[(ii)] For $\alpha+\beta+1=0$,
\begin{equation}\label{entropy dissipation logarithmic}
\begin{split}
\mathcal D[\rho,u]={}&
\frac{\gamma a^2}{2}\int_{\mathbb T}\rho^{\gamma+\alpha-3}|\partial_x\rho|^2\,dx
+\frac12\int_{\mathbb T}\rho^\alpha|\partial_xu|^2\,dx
+\varepsilon\int_{\mathbb T}\rho|u|^2\,dx\\
&+\frac12\int_{\mathbb T}|\partial_{xx}\log\rho|^2\,dx
+\frac{\alpha(1-\alpha)}4
\int_{\mathbb T}|\partial_x\log\rho|^4\,dx.
\end{split}
\end{equation}
\end{itemize}
\end{proposition}
\begin{proof}
The quantity $Q=\mu(\rho)\frac{\partial_x\rho}{\rho^2}$ satisfies
\begin{equation*}
\partial_tQ+u\partial_xQ
=-\frac{1}{\rho}\partial_x\big(\mu(\rho)\partial_xu\big).
\end{equation*}
Consequently, the effective velocity $V=u+Q$ solves
\begin{equation}\label{ rho V main system}
\begin{cases}
\partial_t\rho+\partial_x(\rho u)=0,\\
\mathrm d(\rho V)
+\big[\partial_x(\rho uV)+\partial_xp(\rho)\big]\mathrm dt
=\big[\partial_x\mathcal K-\varepsilon\rho u\big]\mathrm dt
+\rho\sigma\,\mathrm dW.
\end{cases}
\end{equation}
We test the second equation by $V$ and apply It\^o's formula.
The terms involving $u$ give the same internal and capillary energy
derivatives as in Proposition \ref{Prop energy ineq}.
For the remaining pressure contribution, we have
\begin{equation}\label{pressure bd}
\begin{split}
\int_{\mathbb T}\partial_xp(\rho)Q\,dx
&=\gamma a^2\int_{\mathbb T}
\rho^{\gamma+\alpha-3}|\partial_x\rho|^2\,dx=\frac{4\gamma a^2}{(\gamma+\alpha-1)^2}
\int_{\mathbb T}
\left|\partial_x\rho^{\frac{\gamma+\alpha-1}{2}}\right|^2\,dx,
\end{split}
\end{equation}
provided that $\gamma+\alpha-1\ne0$.
When $\gamma+\alpha-1=0$, the last expression is replaced by
$a^2\int_{\mathbb T}|\partial_x\log\rho|^2\,dx$.
Moreover, since $I''(\rho)=\frac{\varepsilon}{2}\rho^{\alpha-2}$,
the continuity equation yields
\begin{equation}\label{damping bd}
\varepsilon\int_{\mathbb T}\rho uQ\,dx
=\frac{d}{dt}\int_{\mathbb T}2I(\rho;\alpha;\varepsilon)\,dx.
\end{equation}
The capillary contribution is handled using the strong coercivity
condition \eqref{SCC}.
If $\alpha+\beta+1\ne0$, integration by parts gives
\begin{equation}\label{capillarity BD}
\begin{split}
-\int_{\mathbb T}\partial_x\mathcal K\,Q\,dx
={}&\frac{4}{(\alpha+\beta+1)^2}
\int_{\mathbb T}
\left|\partial_{xx}\rho^{\frac{\alpha+\beta+1}{2}}\right|^2\,dx+c(\alpha,\beta)\int_{\mathbb T}
\left|\partial_x\rho^{\frac{\alpha+\beta+1}{4}}\right|^4\,dx,
\end{split}
\end{equation}
where $c(\alpha,\beta)$ is defined in \eqref{def c(alpha,beta)}.
Although this constant is not necessarily nonnegative,
Bernis' inequality, see \cite[Proposition~2.8]{Alazard}
and \cite[Lemma~2.1]{Ant}, applied to $\rho^{\frac{\alpha+\beta+1}{2}}$, ensures the nonnegativity of the right-hand side of \eqref{capillarity BD} under \eqref{SCC}.
These elements have already been analysed in several contributions, including \cite{Ant,Le Floch,Pesc}, and we therefore omit the details.
For later use, we also highlight the equivalent representation
\begin{equation}\label{capillarity BD positive}
\begin{split}
-\int_{\mathbb T}\partial_x\mathcal K\,Q\,dx
={}&\frac{9}{(\alpha+\beta+1)^2}
\int_{\mathbb T}
\rho^{\frac{\alpha+\beta+1}{3}}
\left|\partial_{xx}\rho^{\frac{\alpha+\beta+1}{3}}\right|^2\,dx\\
&+\frac{(2\alpha-1-\beta)(\beta-2\alpha+4)}{18}
\int_{\mathbb T}\rho^{\alpha+\beta-3}|\partial_x\rho|^4\,dx,
\end{split}
\end{equation}
which remains valid at the endpoints of \eqref{SCC}.
If $\alpha+\beta+1=0$, a direct computation instead gives
\begin{equation}\label{capillarity BD logarithmic}
\begin{split}
-\int_{\mathbb T}\partial_x\mathcal K\,Q\,dx
={}&\int_{\mathbb T}|\partial_{xx}\log\rho|^2\,dx
+\frac{\alpha(1-\alpha)}{2}
\int_{\mathbb T}|\partial_x\log\rho|^4\,dx.
\end{split}
\end{equation}
In this case, \eqref{SCC} implies $0\le\alpha\le1$,
so the capillary contribution is again nonnegative.
Collecting these identities, integrating in time and taking expectations, we obtain
\begin{equation}\label{BD Entropy ineq Q}
\begin{split}
&\mathbb E\int_{\mathbb T}
\left(
\frac{1}{2}\rho|V|^2+h(\rho)
+\frac{k(\rho)}{2}|\partial_x\rho|^2
+2I(\rho;\alpha;\varepsilon)
\right)(t)\,dx\\
&+\mathbb E\int_0^t\int_{\mathbb T}
\left(
\gamma a^2\rho^{\gamma+\alpha-3}|\partial_x\rho|^2
+\varepsilon\rho|u|^2
-\partial_x\mathcal K\,Q
\right)\,dx\,ds\\
&\le
\int_{\mathbb T}
\left(
\frac{1}{2}\rho_0|V_0|^2+h(\rho_0)
+\frac{k(\rho_0)}{2}|\partial_x\rho_0|^2
+2I(\rho_0;\alpha;\varepsilon)
\right)\,dx
+\frac{1}{2}\|\sigma\|_{L^\infty}^2t,
\end{split}
\end{equation}
from which by recalling \eqref{capillarity BD} and adding the weighted inequalities \eqref{BD Entropy ineq Q}, \eqref{energy inequ} we obtain \eqref{entropy ineq Q/2}.
\end{proof}
\subsection{Pathwise estimates}
The next part of this Section is devoted to a collection of pathwise estimates.  The starting point of this analysis is the following change of variables
\begin{equation*}
\tilde{u}= u- w, \quad w= \sigma W,
\end{equation*}
which allows us to rewrite system \eqref{main system} in terms of $(\rho, \tilde{u})$ satisfying the following random PDE
\begin{equation}\label{random pde}
\begin{cases}
\partial_ t \rho + \partial_x (\rho \tilde{u})= f(\rho,w),  \\
\partial_t (\rho \tilde{u})+ \partial_x(\rho \tilde{u}^2) + \partial_x p(\rho)= \partial_x \mathcal{\tilde{S}}+ \partial_x \mathcal{K}- \varepsilon \rho \tilde{u} + \tilde{u} f(\rho,w)+ \rho g(\rho,\tilde{u},w)
\end{cases}
\end{equation}
where $$ f(\rho,w)=-\partial_x (\rho w), \quad \tilde{S}=\mu(\rho) \partial_x \tilde{u}, \quad g(\rho, \tilde{u},w)=\dfrac{\partial_x(\mu(\rho) \partial_x w)}{\rho} -\varepsilon w-\partial_x(\tilde{u}w)-w \partial_x w,$$
having initial datum 
$(\rho_0,\tilde{u}_0)= (\rho_0,u_0).$
We start with the following pathwise BD entropy estimate. 
\begin{proposition}\label{Prop path BD entropy}
Let $(\rho,u)$ be a solution of \eqref{main system} with initial condition $(\rho_0,u_0) \in \mathcal{X}.$ Assume that \eqref{SCC} holds. Then, for any $T>0$ the following pathwise entropy estimate holds $\mathbb{P}-a.s.$
\begin{equation}\label{entropy ineq path}
\begin{split}
& \sup_{t \in [0,T]}  \mathcal{E}(\rho,u)(t) + \int_{0}^{T} \mathcal{D}[\rho,u]  dt  \le C( \sigma W, a, \mathcal{E}(\rho_0,u_0), T)
\end{split}
\end{equation}
and consequently 
\begin{equation}\label{upper rho path}
\sup_{t \in [0,T]}  \| \rho(t)\|_{L^\infty} \le C(\sigma W, a, \mathcal{E}(\rho_0, u_0), T).
 \end{equation}
Furthermore,  if condition \eqref{NV} is satisfied then the following pathwise lower bound on the density holds
\begin{equation}\label{lower rho path}
\sup_{t \in [0,T]}  \| \rho^{-1}(t) \|_{L^\infty} \le C(\sigma W, a, \mathcal{E}(\rho_0, u_0), T).
\end{equation}
\end{proposition}
\begin{proof}
We test \eqref{random pde} with $\tilde u$ and obtain
\begin{equation}\label{H path}
\frac{d}{dt}H(\rho,\tilde u)
+\|\rho^{\frac{\alpha}{2}}\partial_x\tilde u\|_{L^2}^2
+\varepsilon\|\sqrt{\rho}\tilde u\|_{L^2}^2
=\mathfrak R_1(\rho,\tilde u,f,g),
\end{equation}
where
\begin{equation*}
\begin{split}
\mathfrak R_1(\rho,\tilde u,f,g)
={}&\int_{\mathbb T}
\bigg[
\frac{1}{2}f\tilde u^2+h'(\rho)f
-f\bigg(
\partial_x(k(\rho)\partial_x\rho)
-\frac{k'(\rho)}{2}|\partial_x\rho|^2
\bigg)
+\rho g\tilde u
\bigg]\,dx.
\end{split}
\end{equation*}
Next, we focus on the entropy $\mathcal E(\rho,\tilde u)$.
We observe that the quantity
$Q=\mu(\rho)\frac{\partial_x\rho}{\rho^2}$ satisfies
\begin{equation*}
\begin{split}
\partial_tQ+\tilde u\partial_xQ
&=-\frac{1}{\rho}
\partial_x\big(\mu(\rho)\partial_x\tilde u\big)+r, \qquad 
r=\partial_x\big(\rho^{\alpha-2}f\big)
=-w\partial_xQ
-\frac{1}{\rho}\partial_x\big(\mu(\rho)\partial_xw\big).
\end{split}
\end{equation*}
Thus, setting $\tilde V=\tilde u+Q$, we rewrite the system as
\begin{equation}\label{random pde V}
\begin{cases}
\partial_t\rho+\partial_x(\rho\tilde u)=f,\\
\partial_t(\rho\tilde V)
+\partial_x(\rho\tilde u\tilde V)+\partial_xp(\rho)
=\partial_x\mathcal K-\varepsilon\rho\tilde u
+\tilde V f+\rho g+\rho r.
\end{cases}
\end{equation}
As in Proposition \ref{Prop entropy ineq Q/2},  we multiply the second equation by $\tilde V$ and integrate by parts. Moreover, the continuity equation gives
\begin{equation*}
\frac{d}{dt}\int_{\mathbb T}I(\rho;\alpha;\varepsilon)\,dx
=\frac{\varepsilon}{2}\int_{\mathbb T}\rho\tilde uQ\,dx
+\int_{\mathbb T}I'(\rho;\alpha;\varepsilon)f\,dx.
\end{equation*}
Taking the energy balance \eqref{H path} and the balance for
$\tilde V$ and using
\eqref{entropy with Q}, we obtain
\begin{equation}\label{path energy eq}
\frac{d}{dt}\mathcal E(\rho,\tilde u)
+\mathcal D[\rho,\tilde u]
=\mathfrak R(\rho,\tilde u,f,g,r).
\end{equation}
The contributions already appearing in $\mathfrak R_1$
are contained in $\mathfrak R$. More precisely,
\begin{equation}\label{path remainder decomposition}
\begin{split}
\mathfrak R(\rho,\tilde u,f,g,r)
=\mathfrak R_1+\int_{\mathbb T}
\bigg[
\frac{1}{2}f\tilde uQ+\frac{1}{4}fQ^2
+\frac{1}{2}\rho gQ
+\frac{1}{2}\rho r(\tilde u+Q)
+I'(\rho;\alpha;\varepsilon)f
\bigg]\,dx.
\end{split}
\end{equation}
Our goal is now to estimate $\mathfrak R$ in terms of
the quantities controlled by the entropy.
To this purpose, we first observe that
\begin{equation*}
\begin{split}
\int_{\mathbb T}I'(\rho;\alpha;\varepsilon)f\,dx
&=\int_{\mathbb T}
I''(\rho;\alpha;\varepsilon)\rho\partial_x\rho\,w\,dx=\frac{\varepsilon}{2}
\int_{\mathbb T}\rho Qw\,dx.
\end{split}
\end{equation*}
This term cancels the contribution
$-\frac{\varepsilon}{2}\int_{\mathbb T}\rho Qw\,dx$
coming from $\frac{1}{2}\rho gQ$.
Using the definitions of $f,g,r$ and integrating by parts,
the remaining additional terms in
\eqref{path remainder decomposition} reduce to
\begin{equation*}
\begin{split}
\mathfrak R
=\mathfrak R_1
+\frac{1}{2}\int_{\mathbb T}
\rho^\alpha\partial_x\tilde u\,\partial_xw\,dx
-\frac{1}{2}\int_{\mathbb T}
\rho Q(\tilde u+w)\partial_xw\,dx.
\end{split}
\end{equation*}
Expanding $\mathfrak R_1$, we therefore obtain
\begin{equation}\label{path remainder corrected}
\begin{split}
\mathfrak R
=&-\int_{\mathbb T}p(\rho)\partial_xw\,dx
-\frac{\beta+3}{2}\int_{\mathbb T}
\rho^\beta|\partial_x\rho|^2\partial_xw\,dx-\int_{\mathbb T}
\rho^{\beta+1}\partial_x\rho\,\partial_{xx}w\,dx
\\ & -\frac{1}{2}\int_{\mathbb T}
\rho^\alpha\partial_x\tilde u\,\partial_xw\,dx-\int_{\mathbb T}\rho\tilde u^2\partial_xw\,dx
-\int_{\mathbb T}\rho\tilde u w\partial_xw\,dx\\
&-\frac{1}{2}\int_{\mathbb T}
\rho Q(\tilde u+w)\partial_xw\,dx
-\varepsilon\int_{\mathbb T}\rho\tilde u w\,dx
=\sum_{i=1}^{8}I_i.
\end{split}
\end{equation}
We estimate these terms by using \eqref{entropy elementary controls}.
For the pressure term and the first capillary contribution,
we have
\begin{equation*}
|I_1|+|I_2|
\le C\|\partial_xw\|_{L^\infty}
\big(1+\mathcal E(\rho,\tilde u)\big).
\end{equation*}
To handle $I_3$, we observe that conservation of mass
provides a point $x_0\in\mathbb T$ such that $\rho(x_0)=1$.
If $\beta\ne-2$, the fundamental theorem of calculus yields
\begin{equation}\label{tfc beta}
\begin{split}
\rho^{\frac{\beta+2}{2}}(x)
&=1+\int_{x_0}^{x}
\partial_y\rho^{\frac{\beta+2}{2}}\,dy,\\
\|\rho^{\frac{\beta+2}{2}}\|_{L^\infty}
&\le 1+\|\partial_x\rho^{\frac{\beta+2}{2}}\|_{L^2}\le 1+C_\beta
\left(
\int_{\mathbb T}\rho^\beta|\partial_x\rho|^2\,dx
\right)^{\frac12}.
\end{split}
\end{equation}
Consequently, by H\"older's inequality,
\begin{equation*}
\begin{split}
|I_3|
&\le\|\partial_{xx}w\|_{L^2}
\|\rho^{\beta+1}\partial_x\rho\|_{L^2}\\
&\le C_\beta\|\partial_{xx}w\|_{L^2}
\|\rho^{\frac{\beta+2}{2}}\|_{L^\infty}
\|\partial_x\rho^{\frac{\beta+2}{2}}\|_{L^2}\\
&\le C\|\partial_{xx}w\|_{L^2}
\big(1+\mathcal E(\rho,\tilde u)\big).
\end{split}
\end{equation*}
If $\beta=-2$, instead,
\begin{equation*}
\begin{split}
|I_3|
&=\left|
\int_{\mathbb T}\partial_x\log\rho\,\partial_{xx}w\,dx
\right|\\
&\le\|\partial_{xx}w\|_{L^2}
\|\partial_x\log\rho\|_{L^2}\\
&\le C\|\partial_{xx}w\|_{L^2}
\big(1+\mathcal E(\rho,\tilde u)\big).
\end{split}
\end{equation*}
For the viscous contribution, Young's inequality gives,
for any $\delta>0$,
\begin{equation*}
\begin{split}
|I_4|
&\le\frac{1}{2}
\|\rho^{\frac{\alpha}{2}}\partial_x\tilde u\|_{L^2}
\|\rho^{\frac{\alpha}{2}}\|_{L^2}
\|\partial_xw\|_{L^\infty}\le\delta
\|\rho^{\frac{\alpha}{2}}\partial_x\tilde u\|_{L^2}^2
+C_\delta\|\partial_xw\|_{L^\infty}^2
\big(1+\mathcal E(\rho,\tilde u)\big).
\end{split}
\end{equation*}
The remaining terms are estimated by H\"older's inequality,
the mass constraint and \eqref{entropy elementary controls}:
\begin{equation*}
\begin{split}
|I_5|+|I_6|
&\le\|\partial_xw\|_{L^\infty}
\left(
\|\sqrt{\rho}\tilde u\|_{L^2}^2
+\|w\|_{L^\infty}\|\sqrt{\rho}\tilde u\|_{L^2}
\right),\\
|I_7|
&\le\frac{1}{2}\|\partial_xw\|_{L^\infty}
\|\sqrt{\rho}Q\|_{L^2}
\left(
\|\sqrt{\rho}\tilde u\|_{L^2}+\|w\|_{L^\infty}
\right),\\
|I_8|
&\le\varepsilon\|w\|_{L^\infty}
\|\sqrt{\rho}\tilde u\|_{L^2}.
\end{split}
\end{equation*}
Therefore, choosing $\delta$ sufficiently small and using the embedding
$H^2(\mathbb T)\hookrightarrow W^{1,\infty}(\mathbb T)$, we deduce
from \eqref{path energy eq} that
\begin{equation}\label{path entropy Gronwall}
\frac{d}{dt}\mathcal E(\rho,\tilde u)
+\frac12\mathcal D[\rho,\tilde u]
\le C\big(1+\|w\|_{H^2}^4\big)
\big(1+\mathcal E(\rho,\tilde u)\big).
\end{equation}
Since $w\in C([0,T];H^2(\mathbb T))$ and
$\tilde u(0)=u_0$, Gronwall's lemma yields
\begin{equation*}
\begin{split}
&\sup_{t\in[0,T]}\mathcal E(\rho,\tilde u)(t)
+\int_0^T\mathcal D[\rho,\tilde u]\,dt \le
C\big(w,a,\mathcal E(\rho_0,u_0),T\big).
\end{split}
\end{equation*}
We return to the original variables by using $u=\tilde u+w$.
The decomposition \eqref{entropy with Q}, the mass constraint
and \eqref{entropy elementary controls} imply
\begin{equation*}
\begin{split}
\mathcal E(\rho,u)
&\le C\mathcal E(\rho,\tilde u)+C\|w\|_{L^\infty}^2,\\
\mathcal D[\rho,u]
&\le C\mathcal D[\rho,\tilde u]
+C\big(1+\|w\|_{H^2}^2\big)
\big(1+\mathcal E(\rho,\tilde u)\big).
\end{split}
\end{equation*}
This proves \eqref{entropy ineq path}.
Finally, to deduce \eqref{upper rho path} and \eqref{lower rho path} we observe that by virtue of the conservation of mass and the fundamental theorem of calculus, for any $\theta\in\mathbb{R}$,
\begin{equation}\label{rho L infty alpha}
\begin{split}
\left|\frac{\rho^\theta(x)-1}{\theta}\right|
&=\left|\int_{x_0}^{x}\rho^{\theta-1}\partial_y\rho\,dy\right|
\le\int_{\mathbb{T}}\rho^{\theta-1}|\partial_x\rho|dx,
\end{split}
\end{equation}
where the quotient is understood as $\log\rho$ when $\theta=0$.
On the other hand, since
\begin{equation}\label{visc energ path}
\sqrt{\rho}Q=\rho^{\alpha-\frac{3}{2}}\partial_x\rho
\in L^\infty(0,T;L^2(\mathbb{T})),\quad \mathbb{P}-a.s.,
\end{equation}
by using the mass constraint and H\"older inequality we have
\begin{equation}
\begin{split}
\int_{\mathbb{T}}\rho^{\alpha-1}|\partial_x\rho|\le\|\sqrt{\rho}\|_{L^2}\|\sqrt{\rho}Q\|_{L^2}
\le C(\sigma W,a,\mathcal{E}(\rho_0,u_0),T).
\end{split}
\end{equation}
Then \eqref{upper rho path} follows by taking $\theta=\alpha$ and the supremum in both space and time in \eqref{rho L infty alpha}.
With similar lines of argument, by virtue of \eqref{visc energ path} and the estimate
\begin{equation*}
\rho^{\frac{\beta}{2}}\partial_x\rho
\in L^\infty(0,T;L^2(\mathbb{T})),\quad \mathbb{P}-a.s.,
\end{equation*}
we also deduce \eqref{lower rho path}, provided that \eqref{NV} is satisfied.
Indeed, we take $\theta=\alpha-\frac{1}{2}$ if $\alpha\le\frac{1}{2}$, or $\theta=\frac{\beta+2}{2}$ if $\beta\le-2$ in \eqref{rho L infty alpha}, obtaining a uniform bound on a negative power of $\rho$, or on $\log\rho$ when $\theta=0$.
This concludes the proof.
\end{proof}
\noindent
The next proposition gives higher-order derivative estimates.
The augmented variable $A(\rho)$ reveals the skew-symmetric structure
of \eqref{main system}, providing a fundamental cancellation between
the highest derivatives in the density and velocity equations.
In the zero-capillarity case, the absence of the highest-order
capillary coupling allows the forcing terms to be estimated directly
under an $H^2$ assumption on $\sigma$;
see \cite[Proposition 2.4]{Coti}.
Here, however, the coupling creates an apparent loss of one spatial
derivative in the forcing estimates, whereas \eqref{noise H3}
provides only $w\in C([0,T];H^2(\mathbb T))$.
We overcome this difficulty through a modified energy
$\mathcal E_\eta$, incorporating the mixed term
$\eta\int_{\mathbb T}\partial_xA(\rho)\partial_x\tilde u\,dx$
with a suitably small $\eta>0$.
This provides the additional control of $\partial_{xx}A(\rho)$
needed to close the density and velocity estimates.
\begin{proposition}\label{Prop path HS est}
Let $(\rho,u)$ be a solution of \eqref{main system} with initial condition $(\rho_0,u_0) \in \mathcal{X}.$ 
Assume that conditions \eqref{SCC} and \eqref{NV} are satisfied. Then, for any $T>0$ the following pathwise high-order estimate holds $\mathbb{P}-a.s.$
\begin{equation}\label{H2 estimate}
\sup_{t \in [0,T]} \bigg[ \| \partial_{x} A(\rho) \|^2_{L^2} + \| \partial_x \tilde{u} \|^2_{L^2} \bigg] + \int_{0}^{T} \int_{\mathbb{T}} \rho^{\alpha-1} |\partial_{xx} \tilde{u}|^2  dxdt \le C,
\end{equation}
\begin{equation}\label{dissip A}
\int_{0}^{T} \int_{\mathbb{T}} | \partial_{xx} A(\rho) |^2 dxdt \le C,
\end{equation}
with 
\begin{equation}\label{def A(rho)}
A(\rho)= \sqrt{\dfrac{k(\rho)}{\rho}} \partial_x \rho
\end{equation}
and $C$ being a generic constant $C=C( \sigma W, a, \mathcal{E}(\rho_0,u_0),  \| \partial_{x}  A(\rho_0) \|_{L^2}, \| \partial_x u_0 \|_{L^2}, T).$
\end{proposition}
\begin{proof}
First we observe that by virtue of the introduction of the unknown $A(\rho)$, system \eqref{main system} can be rewritten as
\begin{equation}\label{eq A(rho)}
\partial_t A(\rho)+\partial_x(A(\rho)\tilde{u})
+\partial_x\bigg(\sqrt{\rho k(\rho)}\partial_x\tilde{u}\bigg)
=\partial_x\bigg(\sqrt{\frac{k(\rho)}{\rho}}f\bigg),
\end{equation}
\begin{equation}\label{random pde u systm}
\begin{split}
\partial_t\tilde{u}+\tilde{u}\partial_x\tilde{u}
+\frac{\partial_xp(\rho)}{\rho}
&=\frac{\partial_x(\mu(\rho)\partial_x\tilde{u})}{\rho}
+\partial_x\bigg(\sqrt{\rho k(\rho)}\partial_xA(\rho)\bigg)+A(\rho)\partial_xA(\rho)-\varepsilon\tilde{u}
+g(\rho,\tilde{u},w),
\end{split}
\end{equation}
where we recall that $g(\rho, \tilde{u},w)=\dfrac{\partial_x(\mu(\rho) \partial_x w)}{\rho} -\varepsilon w-\partial_x(\tilde{u}w)-w \partial_x w.$
\\
Then we multiply \eqref{eq A(rho)} by $-\partial_{xx}A(\rho)$ and integrating by parts we get
\begin{equation}\label{test partial xx rho}
\begin{split}
\frac{1}{2}\frac{d}{dt}\|\partial_xA(\rho)\|^2_{L^2}
&=\int_{\mathbb{T}}\partial_x(A(\rho)\tilde{u})
\partial_{xx}A(\rho)dx+\int_{\mathbb{T}}
\partial_x\bigg(\sqrt{\rho k(\rho)}\partial_x\tilde{u}\bigg)
\partial_{xx}A(\rho)dx\\
&\quad-\int_{\mathbb{T}}
\partial_x\bigg(\sqrt{\frac{k(\rho)}{\rho}}f\bigg)
\partial_{xx}A(\rho)dx.
\end{split}
\end{equation}
Similarly, by testing \eqref{random pde u systm} with
$-\partial_{xx}\tilde{u}$ we obtain
\begin{equation}\label{test partial xx u}
\begin{split}
\frac{1}{2}\frac{d}{dt}\|\partial_x\tilde{u}\|^2_{L^2}
+\int_{\mathbb{T}}\rho^{\alpha-1}
|\partial_{xx}\tilde{u}|^2dx & =-\varepsilon\|\partial_x\tilde{u}\|^2_{L^2}
+\int_{\mathbb{T}}\tilde{u}\partial_x\tilde{u}
\partial_{xx}\tilde{u}dx
+\int_{\mathbb{T}}\frac{\partial_xp(\rho)}{\rho}
\partial_{xx}\tilde{u}dx\\
&\quad-\int_{\mathbb{T}}
\frac{\partial_x\mu(\rho)}{\rho}\partial_x\tilde{u}
\partial_{xx}\tilde{u}dx
-\int_{\mathbb{T}}
\partial_x\bigg(\sqrt{\rho k(\rho)}\partial_xA(\rho)\bigg)
\partial_{xx}\tilde{u}dx\\
&\quad-\int_{\mathbb{T}}A(\rho)\partial_xA(\rho)
\partial_{xx}\tilde{u}dx
-\int_{\mathbb{T}}g(\rho,\tilde{u},w)
\partial_{xx}\tilde{u}dx.
\end{split}
\end{equation}
The key point here is to note that by virtue of the skew-symmetric structure of
\eqref{eq A(rho)}--\eqref{random pde u systm},
a cancellation between the highest order capillary terms occurs.
To be precise,
\begin{equation}\label{sum with cancellation}
\begin{split}
&\int_{\mathbb{T}}
\partial_x\bigg(\sqrt{\rho k(\rho)}\partial_x\tilde{u}\bigg)
\partial_{xx}A(\rho)dx
-\int_{\mathbb{T}}
\partial_x\bigg(\sqrt{\rho k(\rho)}\partial_xA(\rho)\bigg)
\partial_{xx}\tilde{u}dx\\
&=\int_{\mathbb{T}}
\partial_x\sqrt{\rho k(\rho)}\,
\partial_x\tilde{u}\partial_{xx}A(\rho)dx
-\int_{\mathbb{T}}
\partial_x\sqrt{\rho k(\rho)}\,
\partial_xA(\rho)\partial_{xx}\tilde{u}dx.
\end{split}
\end{equation}
Thus, summing up \eqref{test partial xx rho} and
\eqref{test partial xx u}, we end up with
\begin{equation}\label{Hs equality final}
\begin{split}
&\frac{1}{2}\frac{d}{dt}
\bigg(\|\partial_xA(\rho)\|^2_{L^2}
+\|\partial_x\tilde{u}\|^2_{L^2}\bigg)
+\int_{\mathbb{T}}\rho^{\alpha-1}
|\partial_{xx}\tilde{u}|^2dx
+\varepsilon\|\partial_x\tilde{u}\|^2_{L^2}\\
=& \int_{\mathbb{T}}\partial_x(A(\rho)\tilde{u})
\partial_{xx}A(\rho)dx
-\int_{\mathbb{T}}
\partial_x\bigg(\sqrt{\frac{k(\rho)}{\rho}}f\bigg)
\partial_{xx}A(\rho)dx\\
&+\int_{\mathbb{T}}\tilde{u}\partial_x\tilde{u}
\partial_{xx}\tilde{u}dx
+\int_{\mathbb{T}}\frac{\partial_xp(\rho)}{\rho}
\partial_{xx}\tilde{u}dx-\int_{\mathbb{T}}
\frac{\partial_x\mu(\rho)}{\rho}\partial_x\tilde{u}
\partial_{xx}\tilde{u}dx \\ &
-\int_{\mathbb{T}}g(\rho,\tilde{u},w)
\partial_{xx}\tilde{u}dx-\int_{\mathbb{T}}A(\rho)\partial_xA(\rho)
\partial_{xx}\tilde{u}dx \\ &
+\int_{\mathbb{T}}
\partial_x\sqrt{\rho k(\rho)}\,
\partial_x\tilde{u}\partial_{xx}A(\rho)dx-\int_{\mathbb{T}}
\partial_x\sqrt{\rho k(\rho)}\,
\partial_xA(\rho)\partial_{xx}\tilde{u}dx
=\sum_{i=1}^{9}I_i.
\end{split}
\end{equation}
Now we start estimating the integrals on the right hand side of
\eqref{Hs equality final}. In particular, we will make an extensive use
of H\"older, Sobolev and Young inequalities.
Throughout the proof, we use the upper and lower bounds on the density provided by Proposition \ref{Prop path BD entropy}.
Additionally, we highlight that with direct computations and Proposition \ref{Prop path BD entropy} we also have
\begin{equation}\label{low order integrated H2}
\begin{split}
\int_0^T\bigg(
\|\partial_{xx}\rho\|^2_{L^2}
+\|\partial_x\rho\|^4_{L^4}
+\|\partial_xA(\rho)\|^2_{L^2}
+\|\partial_x\tilde{u}\|^2_{L^2}
\bigg)dt\le C,
\end{split}
\end{equation}
for a positive constant $C$ depending only on $\|w\|_{C([0,T];H^2)}$, $a$, $\mathcal E(\rho_0,u_0)$, $T$ and the fixed parameters.
\\
\\
In the estimates below, $\theta>0$ denotes the Young parameter
for the viscous contributions, while $\delta>0$ is reserved
for the estimate of $I_2$ and later use. Integrating by parts in $I_1$, we have
\begin{equation}\label{I1}
\begin{split}
|I_1|
&\le C\|\partial_x\tilde{u}\|_{L^\infty}
\|\partial_xA(\rho)\|^2_{L^2}
+\|A(\rho)\|_{L^\infty}
\|\partial_xA(\rho)\|_{L^2}
\|\partial_{xx}\tilde{u}\|_{L^2}\\
&\le\theta
\|\rho^{\frac{\alpha-1}{2}}\partial_{xx}\tilde{u}\|^2_{L^2}
+C(\theta)\|\partial_xA(\rho)\|^4_{L^2}.
\end{split}
\end{equation}
Concerning $I_2$, we observe that by definition of $f$
\begin{equation*}
\sqrt{\frac{k(\rho)}{\rho}}f
=-wA(\rho)-\rho^{\frac{\beta+1}{2}}\partial_xw.
\end{equation*}
Therefore, using
$\partial_x\sqrt{\rho k(\rho)}=\frac{\beta+1}{2}A(\rho)$
and integrating by parts only in the first term, we obtain
\begin{equation}\label{I28 sigma H3}
\begin{split}
I_2
={}&-\frac12\int_{\mathbb T}\partial_xw\,
|\partial_xA(\rho)|^2dx +\frac{\beta+3}{2}\int_{\mathbb T}
A(\rho)\partial_xw\,\partial_{xx}A(\rho)dx\\
&+\int_{\mathbb T}\rho^{\frac{\beta+1}{2}}
\partial_{xx}w\,\partial_{xx}A(\rho)dx.
\end{split}
\end{equation}
Hence, by H\"older, Sobolev and Young inequalities, for every
$\delta>0$ we have
\begin{equation*}
\begin{split}
|I_2|
&\le C\|\partial_xw\|_{L^\infty}
\|\partial_xA(\rho)\|_{L^2}^2 +C\big(\|\partial_xw\|_{L^\infty}
+\|\partial_{xx}w\|_{L^2}\big)
\|\partial_{xx}A(\rho)\|_{L^2}\\
&\le\delta\|\partial_{xx}A(\rho)\|_{L^2}^2
+C(\delta)\big(1+\|w\|_{H^2}^2\big)
\big(1+\|\partial_xA(\rho)\|_{L^2}^2\big).
\end{split}
\end{equation*}
Concerning the other terms in \eqref{Hs equality final}, we have
\begin{equation}
\begin{split}
|I_3|
&\le\theta
\|\rho^{\frac{\alpha-1}{2}}\partial_{xx}\tilde{u}\|^2_{L^2}
+C(\theta)\bigg(1+\|\partial_x\tilde{u}\|^2_{L^2}\bigg)
\|\partial_x\tilde{u}\|^2_{L^2},\\
|I_4|
&\le\theta
\|\rho^{\frac{\alpha-1}{2}}\partial_{xx}\tilde{u}\|^2_{L^2}
+C(\theta),\\
|I_5|
&\le\theta
\|\rho^{\frac{\alpha-1}{2}}\partial_{xx}\tilde{u}\|^2_{L^2}
+C(\theta)\|\partial_xA(\rho)\|^2_{L^2}
\|\partial_x\tilde{u}\|^2_{L^2}.
\end{split}
\end{equation}
By definition of $g(\rho,\tilde{u},w)$ we have
\begin{equation*}
\begin{split}
\|g(\rho,\tilde{u},w)\|^2_{L^2}
&\le C\bigg[
\|w\|^2_{H^2}
\bigg(1+\|\partial_x\tilde{u}\|^2_{L^2}\bigg)
+\|w\|^4_{H^2}
\bigg].
\end{split}
\end{equation*}
Thus,
\begin{equation}
\begin{split}
|I_6|
&\le\theta
\|\rho^{\frac{\alpha-1}{2}}\partial_{xx}\tilde{u}\|^2_{L^2}
+C(\theta)\bigg[
\|w\|^2_{H^2}
\bigg(1+\|\partial_x\tilde{u}\|^2_{L^2}\bigg)
+\|w\|^4_{H^2}
\bigg].
\end{split}
\end{equation}
Finally, using again
$\partial_x\sqrt{\rho k(\rho)}=\frac{\beta+1}{2}A(\rho)$
and integrating by parts in $I_8$, with the same lines of argument as
in \eqref{I1} we obtain
\begin{equation}\label{I9}
\begin{split}
|I_7|+|I_8|+|I_9|
&\le\theta
\|\rho^{\frac{\alpha-1}{2}}\partial_{xx}\tilde{u}\|^2_{L^2}
+C(\theta)\|\partial_xA(\rho)\|^4_{L^2}.
\end{split}
\end{equation}
We set
\begin{equation*}
\begin{split}
a(t)&=1+\|\partial_xA(\rho)\|_{L^2}^2
+\|\partial_x\tilde u\|_{L^2}^2+\|w\|_{H^2}^4,\\
b(t)&=1+\|w\|_{H^2}^4.
\end{split}
\end{equation*}
Summing up \eqref{I1}--\eqref{I9} and choosing $\theta$ small enough, we obtain for every $\delta>0$,
\begin{equation}\label{higher order diagonal}
\begin{split}
&\frac12\frac{d}{dt}\bigg(
\|\partial_xA(\rho)\|_{L^2}^2
+\|\partial_x\tilde u\|_{L^2}^2\bigg)
+c_0\|\rho^{\frac{\alpha-1}{2}}
\partial_{xx}\tilde u\|_{L^2}^2\\
&\qquad\le\delta\|\partial_{xx}A(\rho)\|_{L^2}^2
+C(\delta)b(t)+C(\delta)a(t)\bigg(
\|\partial_xA(\rho)\|_{L^2}^2
+\|\partial_x\tilde u\|_{L^2}^2\bigg).
\end{split}
\end{equation}
Here $c_0>0$ is independent of $\delta$, which will be fixed below.
To control $\partial_{xx}A(\rho)$, we test
\eqref{random pde u systm} with $-\partial_{xx}A(\rho)$
and \eqref{eq A(rho)} with $-\partial_{xx}\tilde u$.
Summing and integrating by parts, using
$\partial_x\sqrt{\rho k(\rho)}=\frac{\beta+1}{2}A(\rho)$,
we obtain
\begin{equation}\label{higher order crossed identity}
\begin{split}
&\frac{d}{dt}\int_{\mathbb T}
\partial_xA(\rho)\partial_x\tilde u\,dx
+\varepsilon\int_{\mathbb T}
\partial_xA(\rho)\partial_x\tilde u\,dx+\int_{\mathbb T}\sqrt{\rho k(\rho)}
|\partial_{xx}A(\rho)|^2dx\\
&=\int_{\mathbb T}\sqrt{\rho k(\rho)}
|\partial_{xx}\tilde u|^2dx
-\int_{\mathbb T}\rho^{\alpha-1}
\partial_{xx}\tilde u\,\partial_{xx}A(\rho)dx\\
&\quad-\int_{\mathbb T}\partial_x\bigg(
\sqrt{\frac{k(\rho)}{\rho}}f\bigg)
\partial_{xx}\tilde u\,dx
+\frac{\beta+1}{2}\int_{\mathbb T}
A(\rho)\partial_x\tilde u\,\partial_{xx}\tilde u\,dx\\
&\quad-\frac{\beta+3}{2}\int_{\mathbb T}
A(\rho)\partial_xA(\rho)\partial_{xx}A(\rho)dx
-\frac32\int_{\mathbb T}
\partial_xA(\rho)|\partial_x\tilde u|^2dx\\
&\quad-\int_{\mathbb T}\bigg[
\frac{\partial_x\mu(\rho)}{\rho}\partial_x\tilde u
-\frac{\partial_xp(\rho)}{\rho}
+g(\rho,\tilde u,w)\bigg]\partial_{xx}A(\rho)dx.
\end{split}
\end{equation}
Then we observe that by the definition of $f$ and the density bounds,
\begin{equation*}
\begin{split}
\bigg\|\partial_x\bigg(
\sqrt{\frac{k(\rho)}{\rho}}f\bigg)\bigg\|_{L^2}^2
\le C\|w\|_{H^2}^2
\big(1+\|\partial_xA(\rho)\|_{L^2}^2\big).
\end{split}
\end{equation*}
Thus, by an extensive use of H\"older and Young
inequalities in \eqref{higher order crossed identity} we deduce
\begin{equation}\label{higher order crossed estimate}
\begin{split}
&\frac{d}{dt}\int_{\mathbb T}
\partial_xA(\rho)\partial_x\tilde u\,dx
+c_1\|\partial_{xx}A(\rho)\|_{L^2}^2\le C_0\|\rho^{\frac{\alpha-1}{2}}
\partial_{xx}\tilde u\|_{L^2}^2+Cb(t) +Ca(t)\bigg(
\|\partial_xA(\rho)\|_{L^2}^2
+\|\partial_x\tilde u\|_{L^2}^2\bigg),
\end{split}
\end{equation}
where $c_1,C_0>0$ are independent of $\delta$. We now fix
$0<\eta\le\min\{\frac12,\frac{c_0}{2C_0}\}$
and then choose $\delta=\eta \frac{c_1}{2}$. Defining
\begin{equation*}
\begin{split}
\mathcal E_\eta(t)
={}&\frac12\bigg(
\|\partial_xA(\rho)\|_{L^2}^2
+\|\partial_x\tilde u\|_{L^2}^2\bigg)+\eta\int_{\mathbb T}
\partial_xA(\rho)\partial_x\tilde u\,dx,
\end{split}
\end{equation*}
we have
\begin{equation*}
\begin{split}
\frac14\bigg(
\|\partial_xA(\rho)\|_{L^2}^2
+\|\partial_x\tilde u\|_{L^2}^2\bigg)
&\le\mathcal E_\eta(t)\le\frac34\bigg(
\|\partial_xA(\rho)\|_{L^2}^2
+\|\partial_x\tilde u\|_{L^2}^2\bigg).
\end{split}
\end{equation*}
Multiplying \eqref{higher order crossed estimate} by $\eta$ and adding to \eqref{higher order diagonal}, we therefore obtain
\begin{equation}\label{higher order Gronwall}
\begin{split}
&\frac{d}{dt}\mathcal E_\eta(t)
+\frac{c_0}{2}\|\rho^{\frac{\alpha-1}{2}}
\partial_{xx}\tilde u\|_{L^2}^2
+\frac{\eta c_1}{2}
\|\partial_{xx}A(\rho)\|_{L^2}^2\le C(\eta)a(t)\mathcal E_\eta(t)+C(\eta)b(t).
\end{split}
\end{equation}
By \eqref{low order integrated H2} and
$w\in C([0,T];H^2(\mathbb T))$, the functions $a,b$
belong to $L^1(0,T), \; \mathbb{P}$-a.s. Thus Gronwall's lemma and integration of \eqref{higher order Gronwall} give \eqref{H2 estimate} and \eqref{dissip A}. This concludes the proof.
\end{proof}
\section{Relative entropy and applications}\label{Sec4}
\noindent
We introduce a relative entropy framework which establishes continuous dependence on the initial data.
The analysis of relative entropy functionals for compressible fluid flows is a widely used method providing stability results, weak-strong uniqueness properties \cite{Giess}, \cite{Feir Nov}, \cite{Feir. Don.} and it also arises in the study of singular limits \cite{Bresch0}, \cite{Feir Nov2}, \cite{Ant5}.  Our analysis extends to the stochastic case and general exponents $\alpha$ and $\beta$ the approach given in \cite{Bresch0}. Throughout this Section we assume that conditions \eqref{SCC} and \eqref{NV} hold, so that the pathwise regularity estimates of Section \ref{Sec3} are satisfied.
\\
\\
We define the energy density $$\eta:= \dfrac{1}{2} \rho u^2+ h(\rho)+ k(\rho) \dfrac{ | \partial_x \rho |^2}{2}$$ and we rewrite it in the new variables $(\rho,m,J)$  in which it is a convex function $$\eta(\rho, m, J)=\dfrac{1}{2} \dfrac{|m|^2}{\rho}+ h(\rho)+ \dfrac{1}{2} \dfrac{ | J |^2}{\rho}, \quad m= \rho u, \quad J= \rho A(\rho).$$
Then, given two solutions $(\rho,u)$ and $(\bar{\rho},\bar{u})$ of \eqref{main system}, we can define the relative entropy functional as
\begin{equation}
H_r(\rho,u,A(\rho) | \bar{\rho}, \bar{u}, A(\bar{\rho})):= \int_{\mathbb{T}} \eta(\rho, m, J | \bar{\rho}, \bar{m}, \bar{J}) dx,
\end{equation}
with 
\begin{equation}
\begin{split}
\eta(\rho, m, J | \bar{\rho}, \bar{m}, \bar{J}) & = \eta(\rho,m,J)- \eta (\bar{\rho},\bar{m},\bar{J})- \eta_\rho (\bar{\rho},\bar{m},\bar{J})(\rho-\bar{\rho}) \\ & -\eta_m (\bar{\rho},\bar{m},\bar{J})(m-\bar{m})-\eta_J (\bar{\rho},\bar{m},\bar{J})(J-\bar{J}).
\end{split}
\end{equation}
After a straightforward computation we have that
\begin{equation}\label{Relative entropy}
H_r(\rho,u,A(\rho) | \bar{\rho}, \bar{u}, A(\bar{\rho}))= \int_{\mathbb{T}} \bigg[ \dfrac{1}{2} \rho(u-\bar{u})^2+ h(\rho | \bar{\rho})+ \dfrac{1}{2} \rho \big( A(\rho)-A(\bar{\rho}) \big)^2 \bigg] dx,
\end{equation}
where 
\begin{equation*}
h(\rho | \bar{\rho})= h(\rho)-h(\bar{\rho})-h'(\bar{\rho})(\rho-\bar{\rho}).
\end{equation*}
We highlight that by assuming $\gamma=1$ and zero capillarity coefficient,  the expression of \eqref{Relative entropy} reduces to the one considered in \cite{Coti}. Furthermore, the functional \eqref{Relative entropy} is not symmetric and the following dichotomy holds 
\begin{equation*}
H_r(\rho,u,A(\rho) | \bar{\rho}, \bar{u}, A(\bar{\rho}))=0 \iff (\rho,u)=( \bar{\rho}, \bar{u}).
\end{equation*}
Our interest is focused on a comparison between two strong solutions with initial data in $\mathcal{X}.$ This analysis leads to the pathwise uniqueness of solutions given in Proposition \ref{Prop uniqueness} and to a stability result determining the continuous dependence with respect to the initial data in Proposition \ref{continuous dependence on the initial data}.  
We start with the following Lemma which provides an upper and lower bound for the relative entropy. 
\begin{lemma}
Let $H_r$ be defined by \eqref{Relative entropy} and suppose that $0<m\le\rho,\bar\rho\le M<\infty.$ Then the following estimates hold
\begin{equation}\label{upper H}
\begin{split}
H_r(\rho,u,A(\rho) | \bar{\rho}, \bar{u}, A(\bar{\rho})) & \le \dfrac{1}{2} \| \rho \|_{L^\infty} \| u-\bar{u} \|^2_{L^2}+ \dfrac{\gamma a^2}{2} \max \big\{\| \rho^{\gamma-2} \|_{L^\infty},  \| \bar{\rho}^{\gamma-2} \|_{L^\infty} \big\} \| \rho-\bar{\rho}\|^2_{L^2} \\ & + \dfrac{1}{2} \| \rho \|_{L^\infty} \| A(\rho)- A(\bar{\rho}) \|^2_{L^2} 
\end{split}
\end{equation}
and
\begin{equation}\label{lower H}
\begin{split}
&\frac12\min\bigg\{m,\gamma a^2\min\{m^{\gamma-2},M^{\gamma-2}\}\bigg\}
\big[\|u-\bar u\|_{L^2}^2+\|\rho-\bar\rho\|_{L^2}^2+\|A(\rho)-A(\bar\rho)\|_{L^2}^2\big]\\
&\qquad\le H_r(\rho,u,A(\rho)|\bar\rho,\bar u,A(\bar\rho)).
\end{split}
\end{equation}
\end{lemma}
\begin{proof}
First we observe that since $h''(s)=\gamma a^2s^{\gamma-2},$ then by Taylor's theorem we have
\begin{equation}\label{Taylor F}
\begin{split}
\frac{\gamma a^2}{2}\min\{\rho^{\gamma-2},\bar\rho^{\gamma-2}\}(\rho-\bar\rho)^2
&\le h(\rho|\bar\rho) \le\frac{\gamma a^2}{2}\max\{\rho^{\gamma-2},\bar\rho^{\gamma-2}\}(\rho-\bar\rho)^2.
\end{split}
\end{equation}
The upper bound \eqref{upper H} follows upon integration. For the lower bound, use $\rho\ge m$ in the kinetic and capillary terms and $h''(s)\ge\gamma a^2\min\{m^{\gamma-2},M^{\gamma-2}\}$ for $s\in[m,M]$.
\end{proof}
\noindent
We continue our analysis with the following uniqueness result. Throughout this section, we consider two different solutions which are defined on the same stochastic basis and are driven by the same Wiener process $W.$ 
\begin{proposition}\label{Prop uniqueness}
Let $(\rho,u)$ and $(\bar{\rho}, \bar{u})$ be two strong solutions of \eqref{main system}, with the same initial condition $(\rho_0,u_0) \in \mathcal{X}.$ Then the following pathwise uniqueness result holds
$$ \mathbb{P} \bigg\{ (\rho(t),u(t))=(\bar{\rho}(t), \bar{u}(t)), \; \forall t \ge 0 \bigg\}=1.$$
\end{proposition}
\begin{proof}
Our goal is to derive an inequality of the form 
\begin{equation*}
\dfrac{d}{dt} H_r(\rho,u,A(\rho) | \bar{\rho}, \bar{u}, A(\bar{\rho}))\le g(t) H_r(\rho,u,A(\rho) | \bar{\rho}, \bar{u}, A(\bar{\rho})),
\end{equation*}
with $g(t)$ integrable in time and to apply a Gronwall argument. Note that although both $(\rho,u)$ and $(\bar{\rho},\bar{u})$ are regular solutions, by virtue of further applications such as the stability result given in Proposition \ref{continuous dependence on the initial data}, we avoid the dependence of $g(t)$ on norms of higher-order derivatives of $(\rho,u).$ In particular, all the norms involving the solution $(\rho,u)$ are controlled in terms of its entropy $\mathcal{E}(\rho,u).$
\\
\\
Fix $T>0$ and choose $m,M$ such that $0<m\le\rho,\bar\rho\le M$ on $[0,T]\times\mathbb T.$ Throughout this section, all regularity statements and pathwise estimates involving the solutions are understood to hold $\mathbb{P}$-a.s. We denote by $C$ a generic positive constant which may depend on $m,M,\alpha,\beta,\gamma,a$ and may vary from line to line. We also write $A=A(\rho)$, $\bar A=A(\bar\rho).$
First we introduce the quantity $\mu'_k(\rho):=\sqrt{\rho k(\rho)}$ so that system \eqref{main system} can be rewritten as
\begin{equation*}
\begin{cases}
\partial_t \rho + \partial_x (\rho u)=0, \\
\text{d}(\rho u) + [\partial_x (\rho u^2)+ \partial_x p(\rho)] \text{d}t= [\partial_x (\mu(\rho) \partial_x u) +\partial_x ( \rho \mu'_k(\rho) \partial_x A(\rho))- \varepsilon \rho u]\text{d}t+ \rho \sigma(x) \text{d}W,\\
\partial_t (\rho A(\rho))+ \partial_x (\rho u A(\rho))=- \partial_x (\rho \mu'_k(\rho)\partial_x u).
\end{cases}
\end{equation*}
Then, we divide by $\rho$ in the momentum equation to deduce
\begin{equation*}
\text{d} u + \big[u \partial_x u + \frac{\partial_x p(\rho)}{\rho} \big]\text{d}t= \frac{1}{\rho} \big[\partial_x (\mu(\rho) \partial_x u) + \partial_x \mathcal{K}-\varepsilon \rho u  \big]\text{d}t+ \sigma(x) \text{d}W,
\end{equation*}
and we notice that the difference $(u-\bar{u})$ satisfies
\begin{equation*}
\partial_t (u-\bar{u}) + u \partial_x u - \bar{u} \partial_x \bar{u}+ \frac{\partial_x p(\rho)}{\rho}-\frac{\partial_x p(\bar{\rho})}{\bar{\rho}}= \frac{\partial_x (\mu(\rho) \partial_x u)}{\rho}- \frac{\partial_x (\mu(\bar{\rho}) \partial_x \bar{u})}{\bar{\rho}}+ \frac{\partial_x \mathcal{K}}{\rho}-\frac{\partial_x \mathcal{\bar{K}}}{\bar{\rho}}-\varepsilon (u-\bar{u}).
\end{equation*}
Thus the stochastic terms cancel. Then we focus on the derivative of $H_r(\rho,u,A(\rho) | \bar{\rho}, \bar{u}, A(\bar{\rho}))$ and we state the following equality
\begin{align}
& \dfrac{d}{dt} H_r(t) +  \varepsilon \ \int_{\mathbb{T}} \rho (u-\bar{u})^2 dx =  \int_{\mathbb{T}} \partial_x (\mu(\rho) \partial_x u)(u-\bar{u})dx - \int_{\mathbb{T}} \partial_x (\mu(\bar{\rho})\partial_x\bar{u}) (u-\bar{u}) \dfrac{\rho}{\bar{\rho}}dx \notag\\
& - \int_{\mathbb{T}} \rho \partial_x \bar{u} (u-\bar{u})^2 dx - \int_{\mathbb{T}}  p(\rho| \bar{\rho}) \partial_x \bar{u} dx  -  \int_{\mathbb{T}} \rho( A-\bar{A})^2 \partial_x \bar{u} dx \notag\\
& -\int_{\mathbb{T}} \rho( {\mu_k}^{''}(\rho) \partial_x \rho -\mu^{''}_k(\bar{\rho}) \partial_x \bar{\rho})  \big[ (A-\bar{A}) \partial_x \bar{u} - (u-\bar{u}) \partial_x \bar{A} \big] dx \notag\\*
& - \int_{\mathbb{T}} \rho( {\mu_k}^{'}(\rho)-\mu^{'}_k(\bar{\rho}))  \big[ (A-\bar{A}) \partial_{xx} \bar{u} - (u-\bar{u}) \partial_{xx} \bar{A} \big]dx, \label{eq rel entropy }
\end{align}
where $p(\rho|\bar\rho):=p(\rho)-p(\bar\rho)-p'(\bar\rho)(\rho-\bar\rho).$ We postpone the proof of \eqref{eq rel entropy }, referring the reader to Appendix \ref{Appendix} for a complete derivation.
Concerning the viscous terms, we notice that after integration by parts
\begin{equation*}
\begin{split}
&\int_{\mathbb T}(u-\bar u)\left[\partial_x(\mu(\rho)\partial_xu)-\frac{\rho}{\bar\rho}\partial_x(\mu(\bar\rho)\partial_x\bar u)\right]dx\\
&=-\int_{\mathbb T}\mu(\rho)|\partial_x(u-\bar u)|^2dx
+\int_{\mathbb T}(u-\bar u)\partial_{xx}\bar u\left[\mu(\rho)-\frac{\rho}{\bar\rho}\mu(\bar\rho)\right]dx\\
&\quad+\int_{\mathbb T}(u-\bar u)\partial_x\bar u\left[\partial_x\mu(\rho)-\frac{\rho}{\bar\rho}\partial_x\mu(\bar\rho)\right]dx.
\end{split}
\end{equation*}
Thus \eqref{eq rel entropy } can be rewritten as
\begin{align}
& \dfrac{d}{dt} H_r(t)  + \varepsilon  \int_{\mathbb{T}} \rho (u-\bar{u})^2dx+ \int_{\mathbb{T}} \mu(\rho) |\partial_x (u-\bar{u})|^2 dx \notag\\
& = \int_{\mathbb{T}} (u-\bar{u}) \partial_{xx} \bar{u} \bigg[ \frac{\bar{\rho} \mu(\rho) - \rho \mu(\bar{\rho})}{\bar{\rho}} \bigg] dx + \int_{\mathbb{T}} (u-\bar{u}) \partial_{x} \bar{u} \bigg[ \frac{\bar{\rho} \partial_x \mu(\rho) - \rho \partial_x \mu(\bar{\rho})}{\bar{\rho}} \bigg] dx \notag\\
& - \int_{\mathbb{T}} \rho \partial_x \bar{u} (u-\bar{u})^2 dx - \int_{\mathbb{T}}  p(\rho| \bar{\rho}) \partial_x \bar{u} dx -\int_{\mathbb{T}} \rho( A-\bar{A})^2 \partial_x \bar{u} dx \notag\\
& - \int_{\mathbb{T}} \rho( {\mu_k}^{''}(\rho) \partial_x \rho -\mu^{''}_k(\bar{\rho}) \partial_x \bar{\rho})  \big[ (A-\bar{A}) \partial_x \bar{u} - (u-\bar{u}) \partial_x \bar{A} \big] dx \notag\\
& - \int_{\mathbb{T}} \rho( {\mu_k}^{'}(\rho)-\mu^{'}_k(\bar{\rho}))  \big[ (A-\bar{A}) \partial_{xx} \bar{u} - (u-\bar{u}) \partial_{xx} \bar{A} \big]dx \notag\\*
& = \sum_{i=1}^{7}  I_i. \label{eq rel entropy visc with sign}
\end{align}
Now we start estimating the integrals in the right-hand side of \eqref{eq rel entropy visc with sign}. In particular, we make an extensive use of the mean value theorem,  H\"{o}lder and Sobolev inequalities and \eqref{lower H}. 
\begin{equation}\label{I1 REL ENTR}
\begin{split}
|I_1|&\le C\|u-\bar u\|_{L^\infty}\|\partial_{xx}\bar u\|_{L^2}\|\rho-\bar\rho\|_{L^2}\\
&\le\delta\|\rho^{\frac{\alpha}{2}}\partial_x(u-\bar u)\|_{L^2}^2
+C(\delta)\big(\|\partial_{xx}\bar u\|_{L^2}+\|\partial_{xx}\bar u\|_{L^2}^2\big)H_r.
\end{split}
\end{equation}
By definition of $A$ we have
\begin{equation}
\begin{split}
I_2&=\alpha\int_{\mathbb T}(u-\bar u)\partial_x\bar u\,\rho
\big[\rho^{\frac{(2\alpha-\beta-3)}{2}}A-\bar\rho^{\frac{(2\alpha-\beta-3)}{2}}\bar A\big]dx,\\
|I_2|&\le C\|\partial_x\bar u\|_{L^\infty}\|u-\bar u\|_{L^2}
\big(\|A-\bar A\|_{L^2}+\|\bar A\|_{L^\infty}\|\rho-\bar\rho\|_{L^2}\big)\\
&\le C\|\partial_x\bar u\|_{L^\infty}(1+\|\bar A\|_{L^\infty})H_r.
\end{split}
\end{equation}
For the remaining terms we have
\begin{equation}
\begin{split}
|I_3|+|I_5|&\le2\|\partial_x\bar u\|_{L^\infty}H_r,\\
|I_4|&\le\int_{\mathbb T}p(\rho|\bar\rho)|\partial_x\bar u|dx
\le(\gamma-1)\|\partial_x\bar u\|_{L^\infty}H_r.
\end{split}
\end{equation}
Since $\partial_x\mu'_k(\rho)=\frac{\beta+1}{2}A(\rho)$, we also have
\begin{equation}
\begin{split}
|I_6|&\le\frac{|\beta+1|}{2}
\big(2\|\partial_x\bar u\|_{L^\infty}+\|\partial_x\bar A\|_{L^\infty}\big)H_r.
\end{split}
\end{equation}
For the last term, the mean value theorem yields
\begin{equation}\label{generalized Bresch}
\begin{split}
\|\mu'_k(\rho)-\mu'_k(\bar\rho)\|_{L^2}
&\le\frac{|\beta+1|}{2}\max\{m^{\frac{(\beta-1)}{2}},M^{\frac{(\beta-1)}{2}}\}
\|\rho-\bar\rho\|_{L^2}.
\end{split}
\end{equation}
Since
\begin{equation*}
\partial_x\big(\mu'_k(\rho)-\mu'_k(\bar\rho)\big)
=\frac{\beta+1}{2}(A-\bar A),
\end{equation*}
combining \eqref{generalized Bresch} with Sobolev embedding gives
\begin{equation*}
\begin{split}
\|\mu'_k(\rho)-\mu'_k(\bar\rho)\|_{L^\infty}
&\le C\big(\|\rho-\bar\rho\|_{L^2}+\|A-\bar A\|_{L^2}\big)
\le C H_r^{1/2}.
\end{split}
\end{equation*}
Consequently,
\begin{equation}\label{I9 REL ENTR}
\begin{split}
|I_7|&\le M\|\mu'_k(\rho)-\mu'_k(\bar\rho)\|_{L^\infty}
\big(\|A-\bar A\|_{L^2}\|\partial_{xx}\bar u\|_{L^2}
+\|u-\bar u\|_{L^2}\|\partial_{xx}\bar A\|_{L^2}\big)\\
&\le C\big(\|\partial_{xx}\bar u\|_{L^2}+\|\partial_{xx}\bar A\|_{L^2}\big)H_r.
\end{split}
\end{equation}
For $\beta=-1$, both $I_6$ and $I_7$ vanish identically. Taking $\delta=\frac{1}{2}$ in \eqref{I1 REL ENTR} and using Poincar\'e inequality we obtain
\begin{equation}\label{Gronw uniqueness}
\begin{split}
&\frac{d}{dt}H_r(t)+\varepsilon\int_{\mathbb T}\rho(u-\bar u)^2dx
+\frac12\int_{\mathbb T}\mu(\rho)|\partial_x(u-\bar u)|^2dx
\le g(t)H_r(t),
\end{split}
\end{equation}
for $$g(t)=C\big[(1+\|\bar A\|_{L^\infty})\|\partial_{xx}\bar u\|_{L^2}
+\|\partial_{xx}\bar u\|_{L^2}^2+\|\partial_{xx}\bar A\|_{L^2}\big].$$
By Proposition \ref{Prop path HS est} and the regularity of $w$,
$g\in L^1(0,T)$, $\mathbb{P}$-a.s.
Hence, since $H_r(0)=0$, Gronwall's lemma gives $H_r(t)=0$ for every $t\in[0,T]$ on an event of probability one. Therefore pathwise uniqueness holds.
\end{proof}
\noindent
Our next goal is to investigate the continuous dependence with respect to the initial data. To this purpose, we use \eqref{Gronw uniqueness} and Proposition \ref{Prop uniqueness} to infer the following result.
\begin{proposition}\label{continuous dependence on the initial data}
Let $\{(\rho^n,u^n)\}_{n\in\mathbb N}$ be a sequence of solutions to \eqref{main system} with initial data $(\rho_0^n,u_0^n)$ and let $(\rho,u)$ be a solution to \eqref{main system} with $(\rho_0,u_0)\in\mathcal X.$ Then
\begin{equation}
(\rho^n(t),u^n(t))\longrightarrow(\rho(t),u(t))
\quad\text{a.s. in }H^1(\mathbb T)\times L^2(\mathbb T),\quad\text{for every }t\ge0,
\end{equation}
provided that
\begin{equation}\label{conv initial data}
(\rho_0^n,u_0^n)\longrightarrow(\rho_0,u_0)
\; \, \text{in } H^1(\mathbb T)\times L^2(\mathbb T).
\end{equation}
\end{proposition}
\begin{proof}
We fix $T>0$ and we restrict to the probability-one event on which \eqref{conv initial data} and the pathwise estimates hold for every solution. Since $H^1(\mathbb T) \hookrightarrow C(\mathbb T)$ and $\rho_0$ is strictly positive, the initial densities have common positive lower and finite upper bounds. It follows from the definition of $\mathcal E$ and \eqref{conv initial data} that
\begin{equation*}
R=\max\left\{\sup_{n\in\mathbb N}\mathcal E(\rho_0^n,u_0^n),\mathcal E(\rho_0,u_0)\right\}<\infty.
\end{equation*}
Therefore, Proposition \ref{Prop path BD entropy} gives
\begin{equation*}
\begin{split}
\sup_{t\in[0,T]}\big(\|\rho(t)\|_{L^\infty}+\|\rho(t)^{-1}\|_{L^\infty}\big)
+\sup_{n\in\mathbb N}\sup_{t\in[0,T]}
\big(\|\rho^n(t)\|_{L^\infty}+\|(\rho^n(t))^{-1}\|_{L^\infty}\big)
\le C(T,\sigma W,R).
\end{split}
\end{equation*}
In particular, \eqref{Gronw uniqueness} yields
\begin{equation}\label{gronw cont dep}
\begin{split}
\frac{d}{dt}H_r(\rho^n,u^n,A(\rho^n)|\rho,u,A(\rho))
&\le g(t)H_r(\rho^n,u^n,A(\rho^n)|\rho,u,A(\rho)),\\
\int_0^Tg(t)dt&\le C\big(T,\sigma W,R,\|\partial_xA(\rho_0)\|_{L^2},\|\partial_xu_0\|_{L^2}\big),
\end{split}
\end{equation}
where $g$ is independent of $n$ and no bounds for $\|\partial_{xx}\rho_0^n\|_{L^2}$ or $\|\partial_xu_0^n\|_{L^2}$ are used. 
\\
Moreover, by using again \eqref{conv initial data}, the definition of $A$
and the uniform initial density bounds, we have
$A(\rho_0^n)\to A(\rho_0)$ in $L^2(\mathbb{T})$.
Therefore, by the upper bound \eqref{upper H}, we infer
\begin{equation*}
\lim_{n\rightarrow\infty}
H_r(\rho_0^n,u_0^n,A(\rho_0^n)|\rho_0,u_0,A(\rho_0))
=0,\quad \mathbb{P}\text{-a.s.},
\end{equation*}
which together with a standard Gronwall lemma applied to
\eqref{gronw cont dep} leads to
\begin{equation*}
\lim_{n\rightarrow\infty}\sup_{t\in[0,T]}
H_r(\rho^n(t),u^n(t),A(\rho^n(t))
|\rho(t),u(t),A(\rho(t)))
=0,\quad \mathbb{P}\text{-a.s.}
\end{equation*}
Finally, \eqref{lower H} yields $L^2$ convergence of
$\rho^n$, $u^n$ and $A(\rho^n)$.
Together with the uniform density bounds, this concludes the proof.
\end{proof}
\section{Invariant measures}\label{Sec5}
\noindent
This Section is devoted to the study of regularity properties of the Markov semigroup $\mathcal{P}_t$ and to the proof of our main result. In particular, we deduce the Feller property and we employ the Krylov-Bogoliubov method to infer the existence of invariant measures.
\subsection{Regularity of the Markov semigroup}
We start our analysis by recalling that by virtue of Theorem
\ref{Thm well posedness}, the transition probabilities introduced in
Section \ref{Sec2} are well-defined provided that \eqref{SCC} and
\eqref{NV} hold. The aforementioned conditions are thus assumed throughout this section.
Furthermore, we observe that, by virtue of Theorem \ref{Thm well posedness} and the regularity estimates of Section \ref{Sec3}, the Markov semigroup introduced in Section \ref{Sec2} satisfies, for every $t\ge0$,
\begin{equation*}
\mathcal{P}_t:\mathcal{M}_b(\mathcal{X}_{1,0})\longrightarrow\mathcal{M}_b(\mathcal{X}_{1,0}).
\end{equation*}
Indeed, for any $\phi\in\mathcal{M}_b(\mathcal{X}_{1,0})$, its boundedness and measurability imply that the function
\begin{equation*}
(\rho_0,u_0)\longmapsto\mathbb{E}\big[\phi(\rho(t;\rho_0),u(t;u_0))\big]
\end{equation*}
is bounded and measurable on $\mathcal{X}_{1,0}$.
On the other hand, although the phase space $\mathcal{X}$ is defined
through the $H^2(\mathbb{T})\times H^1(\mathbb{T})$ regularity,
the continuous dependence with respect to the initial data obtained
in Proposition \ref{continuous dependence on the initial data}
holds in the $H^1(\mathbb{T})\times L^2(\mathbb{T})$ metric.
As recalled in Section \ref{Sec2}, these two metrics generate the
same Borel $\sigma$-algebra on $\mathcal{X}$.
We therefore focus our analysis within the class
$C_b(\mathcal{X}_{1,0})$ and infer the following result.
\begin{lemma}\label{Prop Feller}
Let $\mathcal{P}_t$ be the Markov semigroup associated to system
\eqref{main system}. Then the Feller property holds
\begin{equation}
\mathcal{P}_t:C_b(\mathcal{X}_{1,0})
\longrightarrow C_b(\mathcal{X}_{1,0}).
\end{equation}
\end{lemma}
\begin{proof}
The proof of this result relies on the continuous dependence with
respect to the initial data obtained in Proposition
\ref{continuous dependence on the initial data} and on the boundedness
of the test functions.
Fix $t\ge0,$ we start by considering a sequence of initial data
$\{(\rho_0^n,u_0^n)\}_{n\in\mathbb{N}}\subset\mathcal{X}_{1,0}$
satisfying
\begin{equation*}
\lim_{n\rightarrow\infty}(\rho_0^n,u_0^n)
=(\rho_0,u_0)\quad\text{in }\mathcal{X}_{1,0}
\end{equation*}
and by virtue of Proposition \ref{continuous dependence on the initial data} we have
\begin{equation*}
\begin{split}
\lim_{n\rightarrow\infty}
(\rho(t;\rho_0^n),u(t;u_0^n))
&=(\rho(t;\rho_0),u(t;u_0))
\quad\text{in }\mathcal{X}_{1,0},
\quad\mathbb{P}\text{-a.s.}
\end{split}
\end{equation*}
Then, for a generic $\phi\in C_b(\mathcal{X}_{1,0})$, we define
\begin{equation}
M_\phi:=\sup_{(\rho,u)\in\mathcal{X}_{1,0}}|\phi(\rho,u)|<\infty.
\end{equation}
To infer the Feller property, since $\mathcal{P}_t\phi$ is bounded
and $\mathcal{X}_{1,0}$ is a metric space, it is enough to prove the sequential continuity
\begin{equation*}
\lim_{n\rightarrow\infty}
\mathcal{P}_t\phi(\rho_0^n,u_0^n)
=\mathcal{P}_t\phi(\rho_0,u_0).
\end{equation*}
Therefore we consider
\begin{equation}\label{invariance of G proof}
\begin{split}
&\mathcal{P}_t\phi(\rho_0^n,u_0^n)
-\mathcal{P}_t\phi(\rho_0,u_0)=\mathbb{E}\big[
\phi(\rho(t;\rho_0^n),u(t;u_0^n))
-\phi(\rho(t;\rho_0),u(t;u_0))
\big].
\end{split}
\end{equation}
Since $\phi\in C_b(\mathcal{X}_{1,0})$, we have
\begin{equation*}
\lim_{n\rightarrow\infty}
\big[
\phi(\rho(t;\rho_0^n),u(t;u_0^n))
-\phi(\rho(t;\rho_0),u(t;u_0))
\big]
=0,\quad\mathbb{P}\text{-a.s.},
\end{equation*}
and the absolute value of this difference is bounded by $2M_\phi$.
Thus, by the dominated convergence theorem,
\begin{equation}\label{limit dom conv thm}
\begin{split}
\lim_{n\rightarrow\infty}
\mathbb{E}\big[
\phi(\rho(t;\rho_0^n),u(t;u_0^n))
-\phi(\rho(t;\rho_0),u(t;u_0))
\big]=0,
\end{split}
\end{equation}
which proves our claim.
\end{proof}
\subsection{Krylov-Bogoliubov method}
We apply the Krylov-Bogoliubov method to deduce the existence of invariant measures for the Markov semigroup $\mathcal{P}_t$. In particular, we introduce the class of time averaged measures and we prove the tightness on the phase space $\mathcal{X}$ endowed with the $H^1(\mathbb{T})\times L^2(\mathbb{T})$ metric.
We start introducing the family of time-averaged measures
\begin{equation}\label{time averaged meas}
\nu_T(B)=\frac1T\int_0^T\mathbb{P}\big((\rho(t;\rho_0),u(t;u_0))\in B\big)\,dt,
\end{equation}
which is defined for any $T>0$, $(\rho_0,u_0)\in\mathcal{X}$ and $B\in\mathcal{B}(\mathcal{X}_{1,0})$.
Then, we provide a collection of pointwise bounds which will be used in the proof of the stochastic compactness argument.
\begin{lemma}\label{lemma estimates tight}
Let $\alpha+\beta+1\le0$ and let $\rho$ be a positive scalar function satisfying
\begin{equation*}
\rho \in H^2(\mathbb{T}), \qquad \int_{\mathbb{T}} \rho \,dx=1
\end{equation*}
and define 
\begin{equation*}
\begin{split}
M^2:=\begin{cases}
\displaystyle\int_{\mathbb{T}}\left|\partial_{xx}\rho^{\frac{\alpha+\beta+1}{2}}\right|^2\,dx,
&\alpha+\beta+1<0,\\[2pt]
\displaystyle\int_{\mathbb{T}}|\partial_{xx}\log\rho|^2\,dx,
&\alpha+\beta+1=0.
\end{cases}
\end{split}
\end{equation*}
Then the following bounds hold:
\begin{itemize}
\item If $\alpha+\beta+1=0$, then
\begin{equation}\label{bound tight 1}
e^{-c_pM}\le\rho(x)\le e^{c_pM},\qquad x\in\mathbb T,
\end{equation}
and
\begin{equation}\label{bound tight 2}
\|\partial_x\rho\|_{L^2}^2\le c_p^2M^2e^{2c_pM},\qquad
\|\partial_{xx}\rho\|_{L^2}^2\le Ce^{2c_pM}(M^2+M^4),
\end{equation}
where $c_p$ and $C$ are positive constants depending only on $\mathbb T$.
\\
\item If $\alpha+\beta+1<0$, then, for every $x\in\mathbb T$,
\begin{equation}\label{bound tight 1 neg}
(1+c_pM)^{\frac{2}{\alpha+\beta+1}}\le\rho(x)\le
\begin{cases}
\displaystyle
\left(1+\left\|\partial_x\rho^{\frac{\gamma+\alpha-1}{2}}\right\|_{L^2}\right)^{\frac{2}{\gamma+\alpha-1}},
&\gamma+\alpha-1>0,\\[6pt]
\displaystyle
\exp\big(\|\partial_x\log\rho\|_{L^2}\big),
&\gamma+\alpha-1=0,
\end{cases}
\end{equation}
and
\begin{equation}\label{bound tight 2 neg}
\begin{split}
\|\partial_x\rho\|_{L^2}^2
&\le\frac{4c_p^2}{(\alpha+\beta+1)^2}\|\rho^{1-\alpha-\beta}\|_{L^\infty}M^2,\\
\|\partial_{xx}\rho\|_{L^2}^2
&\le C\bigg(\|\rho^{1-\alpha-\beta}\|_{L^\infty}M^2+\|\rho^{-2(\alpha+\beta)}\|_{L^\infty}M^4\bigg),
\end{split}
\end{equation}
where $c_p$ and $C$ are positive constants depending only on $\mathbb T$ and on $\alpha,\beta, \mathbb T$ respectively.
\end{itemize}
\end{lemma}
\begin{proof}
The mass constraint and continuity give a point $x_0\in\mathbb{T}$
with $\rho(x_0)=1$.
If $\alpha+\beta+1=0$, the fundamental theorem of calculus and
Poincar\'e inequality give
\begin{equation*}
\begin{split}
|\log\rho(x)|
&\le\|\partial_x\log\rho\|_{L^2}
\le c_p\|\partial_{xx}\log\rho\|_{L^2}
=c_pM.
\end{split}
\end{equation*}
This proves \eqref{bound tight 1}.
The identities
\begin{equation*}
\begin{split}
\partial_x\rho&=\rho\partial_x\log\rho,\\
\partial_{xx}\rho
&=\rho\big(\partial_{xx}\log\rho+|\partial_x\log\rho|^2\big),
\end{split}
\end{equation*}
together with $\|\partial_x\log\rho\|_{L^4}\le CM$ prove \eqref{bound tight 2}. With similar lines of argument,  for every $\theta\ne0$, the fundamental theorem of calculus gives
\begin{equation}\label{rho theta tight}
\begin{split}
\rho^\theta(x)
&\le1+\left|\int_{x_0}^x\partial_y\rho^\theta\,dy\right|
\le1+\|\partial_x\rho^\theta\|_{L^2}.
\end{split}
\end{equation}
When $\alpha+\beta+1<0$, we take
$\theta=\frac{\alpha+\beta+1}{2}$ and use Poincar\'e inequality
to obtain the lower bound in \eqref{bound tight 1 neg}.
For the upper bound, we take
$\theta=\frac{\gamma+\alpha-1}{2}$ if this exponent is positive;
if it is zero, we apply the fundamental theorem of calculus
to $\log\rho$. Finally, we have
\begin{equation*}
\begin{split}
\left\|\partial_x\rho^{\frac{\alpha+\beta+1}{2}}\right\|_{L^2}
\le c_pM,\qquad \left\|\partial_x\rho^{\frac{\alpha+\beta+1}{2}}\right\|_{L^4} \le CM
\end{split}
\end{equation*}
and the following identities hold
\begin{equation*}
\begin{split}
\partial_x\rho
&=\frac{2}{\alpha+\beta+1}
\rho^{\frac{1-\alpha-\beta}{2}}
\partial_x\rho^{\frac{\alpha+\beta+1}{2}},\\
\partial_{xx}\rho
&=\frac{2}{\alpha+\beta+1}
\rho^{\frac{1-\alpha-\beta}{2}}
\partial_{xx}\rho^{\frac{\alpha+\beta+1}{2}}+\frac{2(1-\alpha-\beta)}{(\alpha+\beta+1)^2}
\rho^{-\alpha-\beta}
\left|\partial_x\rho^{\frac{\alpha+\beta+1}{2}}\right|^2.
\end{split}
\end{equation*}
Hence we integrate over $\mathbb{T}$ and we obtain both estimates in \eqref{bound tight 2 neg}. This concludes the proof.
\end{proof}
\noindent
Next we infer the following stochastic compactness result.
\begin{proposition}\label{Prop tight}
Assume that \eqref{SCC}, \eqref{NV} hold and that $\alpha+\beta+1\le0.$ Then the family of probability measures $\{\nu_T\}_{T\ge1}$
is tight on $\mathcal{X}_{1,0}$.
\end{proposition}
\begin{proof}
For any $R,S\ge1$ we define the following sets
\begin{equation}
K_R=\bigg\{(\rho,u)\in\mathcal{X}:
\mathcal{D}[\rho,u]\le R^2\bigg\},
\end{equation}
\begin{equation}\label{def set C}
\begin{split}
C_S=\bigg\{(\rho,u)\in\mathcal{X}:\;&
\|u\|^2_{L^2}+\|\partial_xu\|^2_{L^2}
+\|\partial_x\rho\|^2_{L^2}+\|\partial_{xx}\rho\|^2_{L^2}+\|\rho\|_{L^\infty}
+\|\rho^{-1}\|_{L^\infty}\le S\bigg\}.
\end{split}
\end{equation}
The proof of Proposition \ref{Prop tight} strongly relies on
Lemma \ref{lemma estimates tight} and Proposition
\ref{Prop entropy ineq Q/2}.
First, we infer that $C_S$ is compact in $\mathcal{X}_{1,0}$.
To this purpose, we observe that it is bounded in $H^2(\mathbb{T})\times H^1(\mathbb{T})$ and it is closed in $\mathcal{X}_{1,0},$ hence the claim follows by Rellich theorem.
Then we prove the existence of a constant $S(R)$ for which
$K_R\subset C_{S(R)}$.
Throughout the following estimates, the constants depend only on
the fixed parameters and on $\mathbb{T}$.
Given $(\rho,u)\in K_R$, we first observe that
\begin{equation*}
\begin{split}
\|u\|^2_{L^2}
&\le\frac{1}{\varepsilon}\|\rho^{-1}\|_{L^\infty}R^2,\qquad
\|\partial_xu\|^2_{L^2}
\le2\|\rho^{-\alpha}\|_{L^\infty}R^2.
\end{split}
\end{equation*}
Then we distinguish the following cases.
For $\alpha+\beta+1=0$, by
\eqref{capillarity BD logarithmic} we have
\begin{equation*}
\|\partial_{xx}\log\rho\|^2_{L^2}\le2R^2.
\end{equation*}
Therefore, by virtue of Lemma \ref{lemma estimates tight} and the
preceding velocity estimates,
\begin{equation*}
\begin{split}
&e^{-CR}\le\rho(x)\le e^{CR},\qquad x\in\mathbb{T},\\
&\|\partial_x\rho\|^2_{L^2}
+\|\partial_{xx}\rho\|^2_{L^2}
\le C(R^2+R^4)e^{CR},\\
&\|u\|^2_{L^2}+\|\partial_xu\|^2_{L^2}
\le CR^2e^{CR}.
\end{split}
\end{equation*}
Thus, for $R\ge1$ and a sufficiently large constant $C,$ we may choose
\begin{equation*}
S(R)=C(1+R^4)e^{CR}
\end{equation*}
and obtain the inclusion $K_R\subset C_{S(R)}$.
\\
For $\alpha+\beta+1<0$, the pressure term in the expression of 
$\mathcal{D}[\rho,u]$ and Lemma \ref{lemma estimates tight} give
\begin{equation*}
\|\rho\|_{L^\infty}\le
\begin{cases}
\displaystyle
(1+CR)^{\frac{2}{\gamma+\alpha-1}},
&\gamma+\alpha-1>0,\\[6pt]
\displaystyle
e^{CR},
&\gamma+\alpha-1=0.
\end{cases}
\end{equation*}
On the other hand, by \eqref{capillarity BD positive},
\begin{equation*}
\begin{split}
\frac{9}{2(\alpha+\beta+1)^2}
\int_{\mathbb{T}}\rho^{\frac{\alpha+\beta+1}{3}}
\left|\partial_{xx}\rho^{\frac{\alpha+\beta+1}{3}}\right|^2dx
\le R^2.
\end{split}
\end{equation*}
Since the upper bound on $\rho$ bounds
$\rho^{\frac{\alpha+\beta+1}{3}}$ away from zero, we deduce
\begin{equation*}
\left\|\partial_{xx}
\rho^{\frac{\alpha+\beta+1}{3}}\right\|_{L^2}\le C(R).
\end{equation*}
A direct computation also yields
\begin{equation*}
\left\|\partial_{xx}
\rho^{\frac{\alpha+\beta+1}{2}}\right\|_{L^2}\le C(R).
\end{equation*}
Applying again Lemma \ref{lemma estimates tight} and the preceding
velocity estimates, we therefore obtain
\begin{equation*}
\begin{split}
&\|\rho^{-1}\|_{L^\infty}
+\|\partial_x\rho\|^2_{L^2}
+\|\partial_{xx}\rho\|^2_{L^2}
+\|u\|^2_{L^2}
+\|\partial_xu\|^2_{L^2}\le C(R).
\end{split}
\end{equation*}
Thus also in this case there exists a constant $S(R)$
such that $K_R\subset C_{S(R)}$.
Therefore, by using Markov inequality, the nonnegativity of
$\mathcal{E}$ and \eqref{entropy ineq Q/2}, we have
\begin{equation}\label{tight estimate}
\begin{split}
\nu_T(C_{S(R)})
&\ge\nu_T(K_R)
=1-\frac{1}{T}\int_0^T
\mathbb{P}\big[\mathcal{D}[\rho,u]>R^2\big]dt\\
&\ge1-\frac{1}{R^2T}\int_0^T
\mathbb{E}\big[\mathcal{D}[\rho,u]\big]dt \ge1-\bigg[\frac{\mathcal{E}(\rho_0,u_0)}{R^2T}
+\frac{\|\sigma\|^2_{L^\infty}}{2R^2}\bigg].
\end{split}
\end{equation}
Since $T\ge1$, the tightness follows by choosing $R$ sufficiently
large.
\end{proof}
\begin{corollary}\label{coroll conv}
There exist a subsequence $T_j\to\infty$ and a probability measure $\nu$ such that
\begin{equation}
\lim_{j\to\infty}\int_{\mathcal{X}}\phi\,\mathrm{d}\nu_{T_j}
=\int_{\mathcal{X}}\phi\,\mathrm{d}\nu,
\qquad\phi\in C_b(\mathcal{X}_{1,0}).
\end{equation}
\end{corollary}
\begin{proof}
Since we are interested in the limit as $T\rightarrow\infty$,
it is sufficient to consider $T\ge1$.
By Proposition \ref{Prop tight} and Prokhorov's theorem, there exist a
subsequence $T_j\rightarrow\infty$ and a measure $\nu$ such that
\begin{equation}\label{weak conv meas}
\begin{split}
\nu_{T_j}\rightharpoonup\nu
\quad\text{in }\mathfrak{B}(\mathcal{X}_{1,0}).
\end{split}
\end{equation}
Moreover, since $C_{S(R)}$ is closed in $\mathcal{X}_{1,0}$,
by using Portmanteau's theorem and \eqref{tight estimate} we have
\begin{equation}\label{prob meas}
\begin{split}
1\ge\nu(\mathcal{X})
&\ge\nu(C_{S(R)})
\ge\limsup_{j\rightarrow\infty}\nu_{T_j}(C_{S(R)}) \ge1-\frac{\|\sigma\|^2_{L^\infty}}{2R^2}.
\end{split}
\end{equation}
Note that to infer \eqref{weak conv meas} and \eqref{prob meas},
the completeness of the phase space is not required; see the
Theorem 5.1 and Theorem 2.1 in \cite{Billingsley}.
Sending $R\rightarrow\infty$ in \eqref{prob meas}, we obtain
$\nu(\mathcal{X})=1.$ This concludes the proof.
\end{proof}
\noindent
After having constructed the limit measure, the rest of the paper is devoted to the proof of its invariance under the action of the Markov semigroup $\mathcal{P}_t$.
\subsection{Proof of the main result}
Finally we are ready to prove the main result of this paper, Theorem \ref{main theorem}. Assume that conditions \eqref{SCC}, \eqref{NV} and \eqref{T} are satisfied. Then given $(\rho_0,u_0)\in\mathcal{X}$ and the family $\{\nu_T\}_T$ defined as in \eqref{time averaged meas}, by Corollary \ref{coroll conv} there exists a limit measure $\nu\in\mathfrak{B}(\mathcal{X}_{1,0})$.
Now, let $\phi\in C_b(\mathcal{X}_{1,0})$. By virtue of Lemma \ref{Prop Feller} we have that $\mathcal{P}_t\phi\in C_b(\mathcal{X}_{1,0})$ for all $t\ge0$. Furthermore, Corollary \ref{coroll conv} gives
\begin{equation}
\lim_{j\to\infty}\int_{\mathcal{X}}\mathcal{P}_t\phi\,\mathrm{d}\nu_{T_j}
=\int_{\mathcal{X}}\mathcal{P}_t\phi\,\mathrm{d}\nu,\qquad t\ge0,
\end{equation}
from which with direct computation we have
\begin{equation*}
\begin{split}
\int_{\mathcal{X}}\mathcal{P}_t\phi\,\mathrm{d}\nu
&=\lim_{j\rightarrow\infty}
\int_{\mathcal{X}}\mathcal{P}_t\phi\,\mathrm{d}\nu_{T_j}=\lim_{j\rightarrow\infty}
\frac{1}{T_j}\int_0^{T_j}
\mathcal{P}_{t+s}\phi(\rho_0,u_0)\,ds=\lim_{j\rightarrow\infty}
\frac{1}{T_j}\int_t^{T_j+t}
\mathcal{P}_s\phi(\rho_0,u_0)\,ds\\
&=\lim_{j\rightarrow\infty}\bigg(
\frac{1}{T_j}\int_0^{T_j}
\mathcal{P}_s\phi(\rho_0,u_0)\,ds
+\frac{1}{T_j}\int_{T_j}^{T_j+t}
\mathcal{P}_s\phi(\rho_0,u_0)\,ds-\frac{1}{T_j}\int_0^t
\mathcal{P}_s\phi(\rho_0,u_0)\,ds
\bigg)\\
&=\lim_{j\rightarrow\infty}
\int_{\mathcal{X}}\phi\,\mathrm{d}\nu_{T_j}
=\int_{\mathcal{X}}\phi\,\mathrm{d}\nu,
\end{split}
\end{equation*}
which is exactly \eqref{def inv meas}.
Finally, to prove \eqref{ineq inv meas} we recall that the functional $\mathcal{D}$ is nonnegative and lower semicontinuous on $\mathcal{X}_{1,0},$ therefore
\begin{equation*}\begin{split}
\mathcal{D}[\rho,u]\le\liminf_{n\to\infty}\mathcal{D}[\rho_n,u_n].
\end{split}\end{equation*}
Using \eqref{entropy ineq Q/2}, the nonnegativity of $\mathcal{E}$ and Portmanteau's theorem on $\mathcal{X}_{1,0}$, we conclude that
\begin{equation}\label{invariant dissipation bound proof}\begin{split}
\int_{\mathcal{X}}\mathcal{D}[\rho,u]\,\mathrm{d}\nu(\rho,u)
&\le\liminf_{j\to\infty}\int_{\mathcal{X}}\mathcal{D}[\rho,u]\,\mathrm{d}\nu_{T_j}(\rho,u)\\
&\le\liminf_{j\to\infty}\left(\frac{\mathcal{E}(\rho_0,u_0)}{T_j}
+\frac12\|\sigma\|_{L^\infty}^2\right)
=\frac12\|\sigma\|_{L^\infty}^2.
\end{split}
\end{equation}
This concludes the proof of Theorem \ref{main theorem}.
\\
\\
\textbf{Declaration of AI use.} The authors used GPT-6 Astra to correct typographical and grammatical errors, check the presentation and double-check the proofs. The mathematical ideas and strategy of proofs belong to the authors, who take full responsibility for the content of this paper.
\par\smallskip\noindent
\textbf{Data Availability.} The authors declare that data sharing is not applicable to this article as no datasets were generated or analysed.
\par\smallskip\noindent
\textbf{Conflict of interest.} The authors declare that they have no conflict of interest.
\appendix
\section{Relative entropy for viscous Korteweg fluids}\label{Appendix}
\noindent
This Appendix is devoted to the proof of the relative entropy identity \eqref{eq rel entropy }. First we recall some identities and we introduce the quantity $\mu'_k(\rho)$
$$A(\rho)=\sqrt{\dfrac{k(\rho)}{\rho}}\partial_x \rho , \qquad \mu'_k(\rho):=\sqrt{\rho k(\rho)}, \qquad \mu'_k(\rho) \partial_x \rho=\rho A(\rho),$$
so that system \eqref{main system} can be rewritten as 
\begin{equation*}
\begin{cases}
\partial_t \rho + \partial_x (\rho u)=0, \\
\text{d}(\rho u) + [\partial_x (\rho u^2)+ \partial_x p(\rho)] \text{d}t= [\partial_x (\mu(\rho) \partial_x u) +\partial_x ( \rho \mu'_k(\rho) \partial_x A(\rho))- \varepsilon \rho u]\text{d}t+ \rho \sigma(x) \text{d}W\\
\partial_t (\rho A(\rho))+ \partial_x (\rho u A(\rho))=- \partial_x (\rho \mu'_k(\rho)\partial_x u),
\end{cases}
\end{equation*}
which is usually also referred to as the augmented system.
We write $A=A(\rho)$ and $\bar A=A(\bar\rho).$
Next we start computing the derivative of each contribution of the relative entropy 
\begin{equation*}
H_r( \rho, u, A| \bar{\rho}, \bar{u}, \bar{A})= \int_{\mathbb{T}} \dfrac{1}{2} \rho(u-\bar{u})^2 + h(\rho| \bar{\rho})+ \dfrac{1}{2} \rho(A-\bar{A})^2 dx.
\end{equation*}
To compute the derivative of the relative kinetic energy, we first observe that 
\begin{equation*}
\begin{split}
& \partial_t ( u-\bar{u}) +u \partial_x (u-\bar{u})+ \partial_x \bar{u}(u-\bar{u}) =  \bigg( \dfrac{\partial_x (\mu(\rho) \partial_x u)}{\rho}-\dfrac{\partial_x (\mu(\bar{\rho}) \partial_x \bar{u})}{\bar{\rho}} \bigg) \\ &  +\bigg(-\frac{\partial_x p}{\rho}+ \frac{\partial_x \bar{p}}{\bar{\rho}}+ \frac{\partial_x \mathcal K}{\rho}-\frac{\partial_x \overline{\mathcal K}}{\bar{\rho}} \bigg)- \varepsilon (u-\bar{u}).
\end{split}
\end{equation*}
Thus the contributions due to the stochastic force cancel since they are linear in the velocity formulation. By rewriting the Korteweg tensor $\mathcal{K}$ in terms of the new quantity $\mu'_k(\rho),$ we easily deduce the following identity for the relative kinetic energy 
\begin{equation}\label{rel kin energy eq}
\begin{split}
& \dfrac{d}{dt} \int_{\mathbb{T}} \dfrac{1}{2} \rho (u-\bar{u})^2 dx + \varepsilon \int_{\mathbb{T}}\rho (u-\bar{u})^2 dx = \int_{\mathbb{T}} \rho(u-\bar{u})\bigg( \dfrac{\partial_x (\mu(\rho) \partial_x u)}{\rho}-\dfrac{\partial_x (\mu(\bar{\rho}) \partial_x \bar{u})}{\bar{\rho}} \bigg)dx  \\ &  -\int_{\mathbb{T}}\rho \partial_x \bar{u} (u-\bar{u})^2 dx +\int_{\mathbb{T}}\rho(u-\bar{u})\bigg(-\frac{\partial_x p(\rho)}{\rho}+ \frac{\partial_x p(\bar{\rho})}{\bar{\rho}} \bigg) dx \\ & +\int_{\mathbb{T}}\rho(u-\bar{u})\bigg(\frac{\partial_x (\rho \mu'_k(\rho)\partial_x A)}{\rho}-\frac{\partial_x (\bar{\rho} \mu'_k(\bar{\rho}) \partial_x \bar{A})}{\bar{\rho}} \bigg) dx.
\end{split}
\end{equation}
Concerning the pressure term,  with direct computation we have 
\begin{equation}\label{rel p}
\dfrac{d}{dt} \int_{\mathbb{T}} h(\rho| \bar{\rho}) dx = \int_{\mathbb{T}} \rho(u-\bar{u}) \partial_x ( h'(\rho)-h'(\bar{\rho}) )dx- \int_{\mathbb{T}} p(\rho| \bar{\rho}) \partial_x \bar{u}dx,
\end{equation}
and we have that $p(\rho| \bar{\rho})$ satisfies $(\gamma-1)h(\rho| \bar{\rho})=p(\rho| \bar{\rho}).$
Furthermore, we observe that since $\rho h^{''}(\rho)=p'(\rho)$ then the following cancellation occurs
\begin{equation}\label{cancellation p}
\int_{\mathbb{T}} \rho(u-\bar{u}) \partial_x ( h'(\rho)-h'(\bar{\rho}) )dx+\int_{\mathbb{T}} \rho(u-\bar{u})\bigg(-\frac{\partial_x p(\rho)}{\rho}+ \frac{\partial_x p(\bar{\rho})}{\bar{\rho}} \bigg) dx=0.
\end{equation} 
Finally, we focus on the contribution due to the capillarity part.  The difference $(A-\bar{A})$ satisfies
\begin{equation*}
\begin{split}
& \partial_t (A-\bar{A}) +u \partial_x (A-\bar{A})+ \partial_x \bar{A}(u-\bar{u}) =  -\bigg( \dfrac{\partial_x (\rho \mu'_k(\rho) \partial_x u)}{\rho}-\dfrac{\partial_x (\bar{\rho}\mu'_k(\bar{\rho}) \partial_x \bar{u})}{\bar{\rho}} \bigg),
\end{split}
\end{equation*}
thus we have 
\begin{equation}\label{rel kort}
\begin{split}
\dfrac{d}{dt}\int_{\mathbb{T}} \dfrac{1}{2} \rho (A-\bar{A})^2 dx = & - \int_{\mathbb{T}}\rho \partial_x \bar{A} (u-\bar{u})(A- \bar{A})dx \\ & - \int_{\mathbb{T}}\rho(A-\bar{A})\bigg( \frac{\partial_x (\rho \mu'_k(\rho) \partial_x u)}{\rho} - \frac{\partial_x (\bar{\rho} \mu'_k(\bar{\rho}) \partial_x \bar{u})}{\bar{\rho}}\bigg) dx.
\end{split}
\end{equation}
Adding \eqref{rel kin energy eq}, \eqref{rel p} and \eqref{rel kort} and recalling the cancellation \eqref{cancellation p} we deduce the following equality 
\begin{equation}\label{dt Hr not final}
\begin{split}
& \dfrac{d}{dt} H_r( \rho, u, A| \bar{\rho}, \bar{u}, \bar{A})+ \varepsilon \int_{\mathbb{T}}\rho (u-\bar{u})^2 dx= \int_{\mathbb{T}} \rho(u-\bar{u})\bigg( \dfrac{\partial_x (\mu(\rho) \partial_x u)}{\rho}-\dfrac{\partial_x (\mu(\bar{\rho}) \partial_x \bar{u})}{\bar{\rho}} \bigg)dx  \\ &  -\int_{\mathbb{T}}\rho \partial_x \bar{u} (u-\bar{u})^2 dx +\int_{\mathbb{T}}\rho(u-\bar{u})\bigg(\frac{\partial_x (\rho \mu'_k(\rho)\partial_x A)}{\rho}-\frac{\partial_x (\bar{\rho} \mu'_k(\bar{\rho}) \partial_x \bar{A})}{\bar{\rho}} \bigg) dx - \int_{\mathbb{T}} p(\rho| \bar{\rho}) \partial_x \bar{u}dx\\ & - \int_{\mathbb{T}}\rho \partial_x \bar{A} (u-\bar{u})(A- \bar{A})dx  - \int_{\mathbb{T}}\rho(A-\bar{A})\bigg( \frac{\partial_x (\rho \mu'_k(\rho) \partial_x u)}{\rho} - \frac{\partial_x (\bar{\rho} \mu'_k(\bar{\rho}) \partial_x \bar{u})}{\bar{\rho}}\bigg) dx.
\end{split}
\end{equation}
To handle the capillarity contributions we observe that integration by parts gives
\begin{equation*}
\begin{split}
\int_{\mathbb T}(u-\bar u)\partial_x\big(\rho\mu'_k(\rho)\partial_x(A-\bar A)\big)dx
-\int_{\mathbb T}(A-\bar A)\partial_x\big(\rho\mu'_k(\rho)\partial_x(u-\bar u)\big)dx=0.
\end{split}
\end{equation*}
Moreover, since $\mu'_k(\rho)\partial_x\rho=\rho A,$ we obtain
\begin{equation*}
\begin{split}
\partial_x\big(\rho\mu'_k(\rho)\big)
-\frac{\rho}{\bar\rho}\partial_x\big(\bar\rho\mu'_k(\bar\rho)\big)
=\rho\big[A-\bar A+\mu''_k(\rho)\partial_x\rho-\mu''_k(\bar\rho)\partial_x\bar\rho\big].
\end{split}
\end{equation*}
Hence we deduce
\begin{equation}\label{last step rel proof}
\begin{split}
&\int_{\mathbb T}\rho(u-\bar u)\left[
\frac{\partial_x(\rho\mu'_k(\rho)\partial_xA)}{\rho}
-\frac{\partial_x(\bar\rho\mu'_k(\bar\rho)\partial_x\bar A)}{\bar\rho}\right]dx\\
&\quad-\int_{\mathbb T}\rho(A-\bar A)\left[
\frac{\partial_x(\rho\mu'_k(\rho)\partial_xu)}{\rho}
-\frac{\partial_x(\bar\rho\mu'_k(\bar\rho)\partial_x\bar u)}{\bar\rho}\right]dx\\
&=-\int_{\mathbb T}\rho\partial_x\bar u\,(A-\bar A)^2dx
+\int_{\mathbb T}\rho\partial_x\bar A\,(A-\bar A)(u-\bar u)dx\\
&\quad-\int_{\mathbb T}\rho\big(\mu''_k(\rho)\partial_x\rho-\mu''_k(\bar\rho)\partial_x\bar\rho\big)
\big[(A-\bar A)\partial_x\bar u-(u-\bar u)\partial_x\bar A\big]dx\\
&\quad-\int_{\mathbb T}\rho\big(\mu'_k(\rho)-\mu'_k(\bar\rho)\big)
\big[(A-\bar A)\partial_{xx}\bar u-(u-\bar u)\partial_{xx}\bar A\big]dx.
\end{split}
\end{equation}
Therefore, by substituting \eqref{last step rel proof} in \eqref{dt Hr not final} we recover \eqref{eq rel entropy }.
\section{Exponential moment bounds}\label{Appendix B}
\noindent
We conclude with the derivation of an exponential martingale inequality
and the corresponding polynomial and exponential moment bounds.
Although these estimates are not necessary for our analysis,
they provide further information on the long-time behaviour
of the solutions.  They are novel and of independent interest, and we aim to investigate their applications in future research.
\\
\\
To begin with, we establish the following auxiliary lemma. Throughout this appendix, the quantity $\frac{\rho^\alpha-1}{\alpha}$ is understood as $\log\rho$ when $\alpha=0$.
\begin{lemma}\label{capillary martingale control}
Let $\rho \in H^2(\mathbb{T})$ be a positive function with $\int_{\mathbb T}\rho\,dx=1$.
Assume that $\alpha\ge0$ and \eqref{SCC} holds.
Then
\begin{equation}\label{capillary martingale estimate}
\begin{split}
\bigg(
\int_{\mathbb T}
\left|\frac{\rho^\alpha-1}{\alpha}\right|\,dx
\bigg)^2
\le-\frac{9}{2}
\int_{\mathbb T}\partial_x\mathcal K\,Q\,dx.
\end{split}
\end{equation}
\end{lemma}

\begin{proof}
We first observe that, for any positive function $f \in H^2(\mathbb{T})$,
\begin{equation}\label{weighted Poincare martingale}
\begin{split}
\int_{\mathbb T}f|\partial_xf|^2\,dx
\le\frac{1}{2}
\int_{\mathbb T}f|\partial_{xx}f|^2\,dx.
\end{split}
\end{equation}
We now assume that $\alpha+\beta+1\ne0$.
Condition \eqref{SCC} implies
\begin{equation*}
\begin{split}
\frac{3}{2}(\alpha-1)
\le\frac{\alpha+\beta+1}{2}
\le\frac{3}{2}\alpha.
\end{split}
\end{equation*}
Considering separately $\rho\le1$ and $\rho\ge1$ we have
\begin{equation*}
\begin{split}
\left|\frac{\rho^\alpha-1}{\alpha}\right|
\le
\frac{3}{|\alpha+\beta+1|}
(1+\rho)
\left|\rho^{\frac{\alpha+\beta+1}{2}}-1\right|.
\end{split}
\end{equation*}
The same comparison holds for $\alpha=0$ with the logarithmic
expression specified above.
By conservation of mass, there exists $x_0\in\mathbb T$
such that $\rho(x_0)=1$.
Thus, the fundamental theorem of calculus and H\"older's inequality
yield
\begin{equation*}
\begin{split}
\int_{\mathbb T}
\left|\frac{\rho^\alpha-1}{\alpha}\right|\,dx
&\le
\frac{6}{|\alpha+\beta+1|}
\left\|\rho^{\frac{\alpha+\beta+1}{2}}-1\right\|_{L^\infty}\le
\frac{6}{|\alpha+\beta+1|}
\left\|\partial_x\rho^{\frac{\alpha+\beta+1}{2}}\right\|_{L^2}.
\end{split}
\end{equation*}
On the other hand, \eqref{capillarity BD positive}
and \eqref{weighted Poincare martingale} applied to
$\rho^{\frac{\alpha+\beta+1}{3}}$ give
\begin{equation*}
\begin{split}
-\int_{\mathbb T}\partial_x\mathcal K\,Q\,dx
&\ge
\frac{9}{(\alpha+\beta+1)^2}
\int_{\mathbb T}
\rho^{\frac{\alpha+\beta+1}{3}}
\left|
\partial_{xx}\rho^{\frac{\alpha+\beta+1}{3}}
\right|^2\,dx\\
&\ge
\frac{18}{(\alpha+\beta+1)^2}
\int_{\mathbb T}
\rho^{\frac{\alpha+\beta+1}{3}}
\left|
\partial_x\rho^{\frac{\alpha+\beta+1}{3}}
\right|^2\,dx\\
&=
\frac{8}{(\alpha+\beta+1)^2}
\left\|\partial_x
\rho^{\frac{\alpha+\beta+1}{2}}\right\|_{L^2}^2.
\end{split}
\end{equation*}
Combining the last two estimates proves
\eqref{capillary martingale estimate}. Finally, if $\alpha+\beta+1=0$,
condition \eqref{SCC} implies $0\le\alpha\le1$.
Therefore, with similar lines of argument as in the previous case we have 
\begin{equation*}
\begin{split}
\int_{\mathbb T}
\left|\frac{\rho^\alpha-1}{\alpha}\right|\,dx
\le2\|\partial_{xx}\log\rho\|_{L^2}.
\end{split}
\end{equation*}
The conclusion then follows from
\eqref{capillarity BD logarithmic}.
\end{proof}
\noindent
Before stating the exponential martingale inequality,
we introduce the following decomposition of
the entropy dissipation functional:
\begin{equation}\label{dissipation martingale decomposition}
\begin{split}
\mathcal D[\rho,u]
&=\mathcal D_Z[\rho,u]+\mathcal D_R[\rho,u],\\
\mathcal D_Z[\rho,u]
&:=\varepsilon\int_{\mathbb T}\rho|u|^2\,dx
-\frac{1}{2}\int_{\mathbb T}
\partial_x\mathcal K\,Q\,dx,\\
\mathcal D_R[\rho,u]
&:=\frac{1}{2}\int_{\mathbb T}
\rho^\alpha|\partial_xu|^2\,dx
+\frac{\gamma a^2}{2}
\int_{\mathbb T}
\rho^{\gamma+\alpha-3}|\partial_x\rho|^2\,dx.
\end{split}
\end{equation}
We write the capillary contribution as $-\frac{1}{2}\int_{\mathbb T}\partial_x\mathcal K\,Q\,dx$ to include the logarithmic case $\alpha+\beta+1=0$ without distinguishing the two cases in the notation. We also highlight that both $\mathcal D_Z$ and $\mathcal D_R$ are nonnegative
under \eqref{SCC}.
\begin{proposition}\label{Prop esp mart}
Let $(\rho,u)$ be a global solution of \eqref{main system} with initial condition $(\rho_0,u_0)\in\mathcal X$.
Assume that \eqref{SCC} holds. Then, for every $R>0$, the following exponential martingale
inequality holds:
\begin{equation}\label{esp mart ineq}
\begin{split}
&\mathbb P\bigg[
\sup_{t\ge0}\bigg(
\mathcal E(\rho,u)(t)
+\int_0^t
\bigg(
\frac{1}{2}\mathcal D_Z[\rho,u]
+\mathcal D_R[\rho,u]
\bigg)\,ds
-\frac{1}{2}\|\sigma\|_{L^\infty}^2t
\bigg)
-\mathcal E(\rho_0,u_0)\ge R
\bigg]
\le e^{-\gamma_0R},
\end{split}
\end{equation}
where
\begin{equation}\label{def gamma martingale}
\begin{split}
\gamma_0
=\min\left\{
\frac{\varepsilon}{2\|\sigma\|_{L^\infty}^2},
\frac{2}{9\|\partial_x\sigma\|_{L^\infty}^2}
\right\}.
\end{split}
\end{equation}
\end{proposition}
\begin{proof}
The proof strongly relies on the exponential martingale inequality
\begin{equation}\label{Z mart ineq}
\begin{split}
\mathbb P\bigg[
\sup_{t\ge0}
\bigg(Z(t)-\frac{\lambda}{2}\langle Z\rangle(t)\bigg)
\ge R
\bigg]
\le e^{-\lambda R},
\qquad R,\lambda>0,
\end{split}
\end{equation}
which holds for any continuous local martingale
$\{Z(t)\}_{t\ge0}$ with $Z(0)=0$.
The entropy equality reads
\begin{equation*}
\begin{split}
\mathcal E(\rho,u)(t)
+\int_0^t\mathcal D[\rho,u]\,ds
&=\mathcal E(\rho_0,u_0)+Z(t) +\frac{1}{2}\sum_{l=1}^{\infty}
\int_0^t\int_{\mathbb T}
\rho|\sigma_l|^2\,dx\,ds,
\end{split}
\end{equation*}
where
\begin{equation*}
\begin{split}
Z(t)
=\sum_{l=1}^{\infty}\int_0^t
\bigg(
\int_{\mathbb T}
\bigg(
\rho u+\mu(\rho)\frac{\partial_x\rho}{2\rho}
\bigg)\sigma_l\,dx
\bigg)\,\mathrm dW^l
\end{split}
\end{equation*}
is a continuous local martingale, with quadratic variation
\begin{equation*}
\begin{split}
\langle Z\rangle(t)
=\sum_{l=1}^{\infty}\int_0^t
\bigg(
\int_{\mathbb T}
\bigg(
\rho u+\mu(\rho)\frac{\partial_x\rho}{2\rho}
\bigg)\sigma_l\,dx
\bigg)^2\,ds.
\end{split}
\end{equation*}
First, by the mass constraint and H\"older's inequality,
\begin{equation}\label{quadratic cov est 1}
\begin{split}
\sum_{l=1}^{\infty}\int_0^t
\bigg(\int_{\mathbb T}\rho u\sigma_l\,dx\bigg)^2\,ds
\le
\|\sigma\|_{L^\infty}^2
\int_0^t\int_{\mathbb T}\rho|u|^2\,dx\,ds.
\end{split}
\end{equation}
For the density contribution, we integrate by parts:
\begin{equation*}
\begin{split}
\int_{\mathbb T}
\mu(\rho)\frac{\partial_x\rho}{2\rho}\sigma_l\,dx
=-\frac{1}{2}\int_{\mathbb T}
\frac{(\rho^\alpha-1)}{\alpha}\partial_x\sigma_l\,dx.
\end{split}
\end{equation*}
Consequently, Lemma \ref{capillary martingale control}
gives
\begin{equation}\label{quadratic cov est 2}
\begin{split}
\sum_{l=1}^{\infty}\int_0^t
\bigg(
\int_{\mathbb T}
\mu(\rho)\frac{\partial_x\rho}{2\rho}\sigma_l\,dx
\bigg)^2\,ds & \le
\frac{1}{4}\|\partial_x\sigma\|_{L^\infty}^2
\int_0^t
\bigg(
\int_{\mathbb T}
\left|\frac{\rho^\alpha-1}{\alpha}\right|\,dx
\bigg)^2\,ds\\
& \le
\frac{9}{4}\|\partial_x\sigma\|_{L^\infty}^2
\int_0^t
\bigg(
-\frac{1}{2}\int_{\mathbb T}
\partial_x\mathcal K\,Q\,dx
\bigg)\,ds.
\end{split}
\end{equation}
Thus, Young's inequality yields
\begin{equation*}
\begin{split}
\langle Z\rangle(t)
&\le
2\|\sigma\|_{L^\infty}^2
\int_0^t\int_{\mathbb T}\rho|u|^2\,dx\,ds+
\frac{9}{2}\|\partial_x\sigma\|_{L^\infty}^2
\int_0^t
\bigg(
-\frac{1}{2}\int_{\mathbb T}
\partial_x\mathcal K\,Q\,dx
\bigg)\,ds.
\end{split}
\end{equation*}
By the definition of $\gamma_0$, we therefore obtain
\begin{equation}\label{quadratic variation dissipation}
\begin{split}
\gamma_0\langle Z\rangle(t)
\le\int_0^t\mathcal D_Z[\rho,u]\,ds.
\end{split}
\end{equation}
To conclude, we define
\begin{equation*}
\begin{split}
\psi(t)
:=\mathcal E(\rho,u)(t)
+\int_0^t
\bigg(
\mathcal D_R[\rho,u]
+\frac{1}{2}\mathcal D_Z[\rho,u]
\bigg)\,ds.
\end{split}
\end{equation*}
Using \eqref{quadratic variation dissipation} we deduce
\begin{equation*}
\begin{split}
\psi(t)-\frac{1}{2}\|\sigma\|_{L^\infty}^2t
&\le
\mathcal E(\rho_0,u_0)+Z(t)
-\frac{1}{2}\int_0^t\mathcal D_Z[\rho,u]\,ds \le
\mathcal E(\rho_0,u_0)+Z(t)
-\frac{\gamma_0}{2}\langle Z\rangle(t).
\end{split}
\end{equation*}
Consequently, applying \eqref{Z mart ineq}, we have
\begin{equation*}
\begin{split}
&\mathbb P\bigg[
\sup_{t\ge0}
\bigg(
\psi(t)-\frac{1}{2}\|\sigma\|_{L^\infty}^2t
\bigg)
-\mathcal E(\rho_0,u_0)\ge R
\bigg]\le
\mathbb P\bigg[
\sup_{t\ge0}
\bigg(
Z(t)-\frac{\gamma_0}{2}\langle Z\rangle(t)
\bigg)\ge R
\bigg]
\le e^{-\gamma_0R},
\end{split}
\end{equation*}
which proves \eqref{esp mart ineq}.
\end{proof}

\noindent
We conclude this part with the following result,
which is a direct consequence of Proposition \ref{Prop esp mart}. The proof is therefore omitted. We refer the reader to \cite{Coti}, Section 2.1.4, for a similar analysis.
\begin{proposition}\label{Prop moment}
Under the hypotheses of Proposition \ref{Prop esp mart},
for every $m\ge1$ there exists a constant $c_m>0$,
depending only on $m$, such that
\begin{equation}\label{moment}
\begin{split}
&\mathbb E\bigg\{
\bigg[
\sup_{t\in[0,T]}
\bigg(
\mathcal E(\rho,u)(t)
+\int_0^t
\bigg(
\frac{1}{2}\mathcal D_Z[\rho,u]
+\mathcal D_R[\rho,u]
\bigg)\,ds
\bigg)
\bigg]^m
\bigg\}\\
&\qquad\le
c_m\bigg(
\mathcal E(\rho_0,u_0)^m
+\|\sigma\|_{L^\infty}^{2m}T^m
+\gamma_0^{-m}
\bigg),
\qquad T\ge0.
\end{split}
\end{equation}
Moreover, the following exponential moment bound holds:
\begin{equation}\label{exp moment}
\begin{split}
&\mathbb E\exp\bigg[
\frac{\gamma_0}{2}
\sup_{t\in[0,T]}
\bigg(
\mathcal E(\rho,u)(t)
+\int_0^t
\bigg(
\frac{1}{2}\mathcal D_Z[\rho,u]
+\mathcal D_R[\rho,u]
\bigg)\,ds
\bigg)
\bigg]\\
&\qquad\le
2\exp\bigg[
\frac{\gamma_0}{2}
\bigg(
\mathcal E(\rho_0,u_0)
+\frac{1}{2}\|\sigma\|_{L^\infty}^2T
\bigg)
\bigg],
\qquad T\ge0.
\end{split}
\end{equation}
\end{proposition}

\end{document}